\documentclass{amsart}

\usepackage[latin1]{inputenc}
\usepackage{amsfonts}
\usepackage{amsmath}
\usepackage{amsthm}
\usepackage{amssymb}
\usepackage{latexsym}
\usepackage{enumerate}

\newtheorem{theorem}{Theorem}
\newtheorem{lemma}[theorem]{Lemma}
\newtheorem{proposition}[theorem]{Proposition}
\newtheorem{corollary}[theorem]{Corollary}

\newtheorem{question}[theorem]{Question}

\newtheorem{definition}[theorem]{Definition}

\numberwithin{equation}{section}

\begin{document}

\newcommand{\cc}{\mathfrak{c}}
\newcommand{\N}{\mathbb{N}}
\newcommand{\C}{\mathbb{C}}
\newcommand{\Q}{\mathbb{Q}}
\newcommand{\R}{\mathbb{R}}
\newcommand{\T}{\mathbb{T}}
\newcommand{\seqn}{\{0,1\}^n}
\newcommand{\In}{I_n}
\newcommand{\st}{*}
\newcommand{\PP}{\mathbb{P}}
\newcommand{\lin}{\left\langle}
\newcommand{\rin}{\right\rangle}
\newcommand{\SSS}{\mathbb{S}}
\newcommand{\forces}{\Vdash}
\newcommand{\dom}{\text{dom}}
\newcommand{\osc}{\text{osc}}
\newcommand{\F}{\mathcal{F}}
\newcommand{\A}{\mathcal{A}}
\newcommand{\B}{\mathcal{B}}
\newcommand{\I}{\mathcal{I}}
\newcommand{\CC}{\mathcal{C}}

\author{Piotr Koszmider}
\address{Institute of Mathematics, Polish Academy of Sciences,
ul. \'Sniadeckich 8,  00-656 Warszawa, Poland}
\email{\texttt{piotr.koszmider@impan.pl}}
\thanks{The author would like to thank the anonymous referee for many helpful suggestions improving the presentation and for the idea of how to fix a gap in the previous version of the proof of Proposition \ref{cover-hyper}.}
\thanks{The author would like to thank G. Plebanek for helpful comments.}

\subjclass[2010]{}
\title[Overcomplete sets in some classical nonseparable
Banach spaces]{On the existence 
of overcomplete sets in some classical nonseparable
Banach spaces}

\maketitle

\begin{abstract} For a  Banach space $X$
its subset $Y\subseteq X$  is called overcomplete
if $|Y|=dens(X)$ and $Z$ is linearly dense in $X$ for every $Z\subseteq Y$ 
with $|Z|=|Y|$. In the context of nonseparable Banach spaces this notion
was introduced recently by T. Russo and J. Somaglia but overcomplete sets have been
 considered in separable Banach spaces since the 1950ties.

We prove some absolute and consistency results concerning the existence and
the nonexistence
of overcomplete sets in some classical nonseparable Banach spaces. 
For example: $c_0(\omega_1)$,  $C([0,\omega_1])$,  $L_1(\{0,1\}^{\omega_1})$,
$\ell_p(\omega_1)$, $L_p(\{0,1\}^{\omega_1})$
for $p\in (1, \infty)$ or in general
WLD Banach spaces of density $\omega_1$ admit overcomplete sets (in ZFC).
The spaces $\ell_\infty$, $\ell_\infty/c_0$, spaces of the form $C(K)$ for
$K$ extremally disconnected, superspaces of $\ell_1(\omega_1)$ of
density $\omega_1$  do not admit overcomplete sets (in ZFC).
Whether the Johnson-Lindenstrauss space generated
in $\ell_\infty$  by $c_0$ and the characteristic functions of 
elements of an  almost disjoint family of subsets of $\N$ of cardinality $\omega_1$
admits an overcomplete set is undecidable. The same refers to
all nonseparable Banach spaces with the dual balls of density $\omega_1$ which are separable in the
weak$^*$ topology.
The results proved refer to wider classes of Banach spaces but several
natural open questions remain open.
\end{abstract}

\section{Introduction}

All Banach spaces considered in this paper are infinite dimensional and over the reals.
The density $dens(X)$ of a Banach space $X$  is the minimal cardinality of a norm dense subset
of $X$. Other terminology and notation used in the introduction can be found in Section \ref{terminology}.

\begin{definition}[\cite{russo}] Let $X$ be an infinite dimensional  Banach space.
A set $Y\subseteq X$  is called overcomplete (in $X$)
if $|Y|=dens(X)$ and $Z$ is linearly dense in $X$ for every $Z\subseteq Y$ 
with $|Z|=|Y|$.
\end{definition}

Overcomplete sets have been investigated for separable Banach spaces as
overcomplete sequences, hypercomplete sequences or overfilling sequences.
They appear in research related to basic sequences in general
Banach spaces (e.g. \cite{terenzi-nobasic, terenzi-sln, plans, singer}). In particular, Terenzi
has proved in \cite{terenzi-nobasic} a remarkable dichotomy concerning sequences with no basic subsequences which involves sequences
overcomplete in their span. The structure of overcomplete sequences has been investigated in \cite{terenzi-structure}.
For other aspects of overcomplete sequences see \cite{partington1, partington2, fonf1, fonf2}.
The following existence, nonexistence and consistency results have been obtained so far:
\begin{itemize}
\item Every separable Banach space admits an overcomplete set (\cite{klee}).
\item ({\sf CH}) $X$ admits an overcomplete set  if $dens(X^*)=\omega_1$ (\cite{russo}).
\item A Banach space $X$ does not admit an overcomplete set if
\begin{itemize}
\item $X=\ell_1(\omega_1)$  (\cite{russo}). 
\item $cf(dens(X))>\mathfrak c$ (\cite{russo}).
\item $dens(X)>\omega_1$ and $X$ admits a fundamental biorthogonal system, in particular, if $X$
admits a Markushevich basis or is WLD (\cite{russo}).
\end{itemize}
\item ({\sf $\neg$CH}) $\ell_\infty$ does not admit an overcomplete set  (\cite{russo}).
\end{itemize}

The purpose of  this paper is to present further existence and nonexistence  results which can be divided into positive, negative,
consistency and independence results\footnote{Recall that {\sf ZFC} is a standard axiomatization
of mathematics (see e.g. \cite{jech, kunen}). A statement $\phi$ is said to be 
consistent (with {\sf ZFC}) if adding it to {\sf ZFC} does not lead to a contradiction
and is said to be independent from {\sf ZFC} or undecidable if both
$\phi$ and the negation of $\phi$ are consistent.}.

Among positive results in Theorem \ref{main-positive} we prove in {\sf ZFC} 
(i.e., without any extra set-theoretic assumptions) that the
following Banach spaces admit overcomplete sets:
\begin{itemize}
\item Every WLD Banach space of density $\omega_1$, in particular
\begin{enumerate}
\item $\ell_p(\omega_1)$, $L_p(\{0,1\}^{\omega_1})$ for $p\in (1,\infty)$. 
\item $L_1(\{0,1\}^{\omega_1})$.
\item $c_0(\omega_1)$.
\item $C(K)$s for $K$ a Corson compact where all Radon measure have separable supports.
\end{enumerate}
\item $C([0,\omega_1])$.
\item $C(K)$, where $K$ is the one point compactification of
a refinement of the order topology on $[0,\omega_1)$ obtained by isolating all points of some  
subset of $[0,\omega_1)$.
\end{itemize}
Note that these are the first results showing the existence in ZFC of nonseparable Banach spaces
admitting overcomplete sets. Also it follows  that it is consistent
that there are Banach spaces $X$ with
$dens(X^*)>\omega_1$ which admit overcomplete sets. Namely assuming
the negation of {\sf CH}, consider $L_1(\{0,1\}^{\omega_1})$ or
$\ell_2(\omega_1)\oplus\ell_1$.
On the other hand we extend the list from \cite{russo}
of Banach spaces which do not admit overcomplete sets in ZFC 
(i.e., without any extra set-theoretic assumptions) and include there
the following:

\begin{itemize}
\item $X$s of density $\omega_1$ which contain $\ell_1(\omega_1)$  (Theorem \ref{main-negative}).
\item $C(K)$s for $K$ compact, Hausdorff, infinite and extremally disconnected (Theorem \ref{main-negative}).
\item $\ell_\infty$, $\ell_\infty/c_0$, $L_\infty(\{0,1\}^{\omega_1})$ 
 (Theorem \ref{main-negative}). 
\item $C([0,1]^{\omega_1})$, $C(\{0,1\}^{\omega_1})$,    (Theorem \ref{main-negative}).
\item 
$C(K)$ which is  Grothendieck space of density $\omega_1$ (Theorem \ref{ck-groth}).
\item Banach space $X$ of density $\kappa$, where  $\kappa$ is a cardinal satisfying $cf(\kappa)>\omega_1$ and 
\begin{itemize}

\item $X$ contains an isomorphic copy of  $\ell_1(\omega_1)$ (Theorem \ref{bigger-l1}).
\item    $X^*$
contains a nonseparable WLD subspace (Theorem \ref{wld-dual}).
\item  $X$
is a nonreflexive Grothendieck space (Theorem \ref{bigger-groth}).
\item  
$X=C(K)$  for $K$ compact, Hausdorff and scattered  (Theorem \ref{bigger-scat}).
\end{itemize}
\end{itemize}
It should be noted that the nonexistence of overcomplete sets in  Banach spaces
$X$ which contain $\ell_1(dens(X))$ for
$dens(X)\geq\omega_2$ can be directly concluded from Theorem 3.6 of of \cite{russo}
(see the remarks after the proof of Theorem \ref{main-negative})). This argument covers
all Banach spaces of the form $C(K)$ for $K$ Hausdorff, compact, infinite and extremally disconnected 
of densities bigger or equal to $\omega_2$, $C([0,1]^\kappa)$, $C(\{0,1\}^\kappa)$ for $\kappa\geq\omega_2$,
$\ell_\infty(\lambda)$, $\ell_\infty(\lambda)/c_0(\lambda)$, $L_\infty(\{0,1\}^\lambda)$ 
for $\lambda\geq\omega_1$. Note, for example, that the case of $\ell_\infty$ having the density equal to
continuum does not follow from Theorem 3.6 of of \cite{russo} if {\sf CH} holds.
Our Theorem \ref{main-negative} has a uniform proof for all densities of uncountable cofinality, in particular 
for the density of continuum whose value depend on {\sf CH}.

We also obtain the following consistency results:
\begin{itemize}
\item {\rm(}{\sf MA+$\neg$CH}{\rm)} A Banach space $X$ does not admit an overcomplete set if
\begin{itemize} 
\item $\omega<cf(dens(X))\leq dens(X)<\mathfrak c$ and
 $B_{X^*}$ is
separable in the weak$^*$ topology  (Theorem \ref{main-consistency}). 
\item $dens(X)=\omega_1$ and  $B_{X^*}$  is not monolithic in the
weak$^*$ topology (Theorem \ref{nonmonolithic}). 
\end{itemize}

\item  It is consistent with {\sf MA} for partial orders having precaliber $\omega_1$
and the negation of {\sf CH} that every Banach space whose dual has density $\omega_1$
admits an overcomplete set (Theorem \ref{ma-precaliber}).

\item The statement  that every Banach space whose dual has density $\omega_1$
admits an overcomplete set is consistent with any size of the continuum (Theorem \ref{any-c}).

\item  {\rm (}$\mathfrak p=\mathfrak c>\omega_1${\rm )} No nonreflexive Grothendieck
 space of regular density (in particular equal to $\mathfrak c$ under the above assumption 
 $\mathfrak p=\mathfrak c$)
 admits an overcomplete set (Corollary \ref{argyros-ma}).
\end{itemize}

Based on the above we conclude a couple of independence results:
\begin{itemize}
\item The
existence of overcomplete sets is independent for 
all  Banach spaces $X$ satisfying:
$dens(X)=dens(X^*)=\omega_1$ and
$B_{X^*}$ is not monolithic in the weak$^*$ topology, in particular such that
$B_{X^*}$ is weakly$^*$ separable.
\item The
existence of a Banach space $X$
admitting an overcomplete set and satisfying: $dens(X)=\omega_1$ and 
 $L_1(\{0,1\}^{\omega_1})\subseteq X^*$ is independent (Corollary \ref{argyros-omega1}).
\end{itemize}

A classical example of  Johnson and Lindenstrauss of a Banach space $X$ satisfying:
$dens(X)=dens(X^*)=\omega_1$ and
$B_{X^*}$ is weakly$^*$ separable  is  the Banach space generated in
$\ell_\infty$ by $c_0$ and $\{1_A: A\in \A\}$, where $\A$ is an almost disjoint family of
 subsets of $\N$ of cardinality $\omega_1$. So in particular, the existence of
 overcomplete sets in such spaces is independent  by Corollary \ref{cor-psi} and the 
 ${\sf CH}$ result of \cite{russo}.

Corollaries of the above results include: 
\begin{itemize}
\item  A WLD Banach space $X$ admits an overcomplete set if and only if
the density of $X$ is less or equal to $\omega_1$ (Corollary \ref{wld-iff}).
\item 
A Banach space $X$ of density $\omega_1$ with an unconditional basis
admits an overcomplete set if and only if $X$ is WLD (Corollary \ref{unconditional}).

\item 
 If $X$ is a Banach spaces such that $cf(dens(X))>\omega$, $dens(X)>\omega_1$
and $L_1(\{0,1\}^{dens(X)})\subseteq X^*$, then $X$ does not admit an 
overcomplete set (Corollary \ref{argyros-bigger}).

\item  If $\kappa$ is an infinite cardinal, then 
$C([0,\kappa])$ admits an overcomplete set if and only if $\kappa\leq\omega_1$ 
(Theorem \ref{main-positive} (2) and result of \cite{russo} i.e., 
Theorem \ref{russo} (3) in this paper).

\end{itemize}

We explain briefly the structure of the paper and the methods used. In Section 2 we 
establish terminology, remind known results and prove some general facts. 
Section 3 is devoted to positive results. They are obtained in Theorem \ref{main-positive}
which is proved by stepping-up the original construction of Klee with the help of
a sequence of coherent injections from countable ordinals into $\N$.

 Section 4 contains
consistency results involving Martin's axiom and simple finite support iterations as well
as the Cohen model. The main ingredient is Proposition \ref{cover-hyper} where
it is proved under {\sf MA}+$\neg${\sf CH} that if $D=\{x_\xi: \xi<\kappa\}\subseteq X$,
where $X$ is a Banach space with weakly$^*$ separable dual ball and $\kappa<\mathfrak c$
and  $x_\xi\not\in \overline{lin}\{x_\eta: \eta<\xi\}$ for any $\xi<\kappa$,
then $D$ can be covered by countably many hyperplanes. We need the hypothesis
on $D$ as in ZFC in any separable $X$ there is $D\subseteq X$ of cardinality $\omega_1$
which cannot be covered
by countably many hyperplanes. Indeed, using the original method of Klee (see the proof of
Theorem \ref{klee}) in a separable Banach space one can construct a set of cardinality
$\mathfrak c$ where every infinite subset is linearly dense. 

Section 5 is devoted to negative results which follow
from the existence of linearly independent 
functionals $\phi, \psi\in X^*$ which assume single values $r, s\in \R$ on
big subsets of a given linearly dense set. Then $s\phi-r\psi$ defines a hyperplane
including a big subset of a linearly dense set. This is Lemma \ref{single-value} 
which is the main tool of that section. Its hypothesis is that the dual sphere
$S_{X^*}$ has many points of character (with respect to the
weak$^*$ topology) equal to the density of $X$. Characters of functionals as points have nice
interpretations for $C(K)$ spaces as types of uniform regularity
of Radon  measures (\cite{pol}, \cite{krupski}). In fact our proof of Lemma
\ref{single-value} is inspired by the methods of \cite{krupski}. To make the main
conclusions in Theorem \ref{main-negative} we need a dense range linear operator
from the space into a space where all characters are big, this is achieved in Lemma \ref{lemma-completion}.
In Section 6 we use counting arguments (e.g. like in  Lemma \ref{unions})
to obtain negative results for Banach spaces $X$ such that $cf(dens(X))>\omega_1$.
The last section discusses unanswered questions. Some of them are quite fundamental.

\section{Preliminaries}

\subsection{Notation and terminology}\label{terminology}

The notation and terminology should be fairly standard.  

$f|A$ denotes the restriction of
a function $f$ to the set $A$. $1_A$ will denote the characteristic function
of a set $A$ (relative to some superset given in the context). $\N$
stands for the non-negative integers. Sometimes $n\in \N$ is
identified with the set $\{0, ..., n-1\}$.
For $n\in \N$ by $\omega_n$ we denote the $n$-th infinite cardinal, $\mathfrak c$ stands for
the cardinality of the continuum, i.e., $\R$. $cf(\xi)$ denotes the cofinality
of an ordinal $\xi$.
$\R$ denotes the reals, $\Q$ denotes the rationals  and
$\Q_+$ denotes the positive rationals. For a set $A$
by $[A]^{2}$ we mean the collection of all two-element subsets of $A$.

All Banach spaces considered in this paper are infinite dimensional and over the reals. $X^*$ stands for the dual space of $X$. $B_{X}$ and $S_X$ stand for the unit ball
and the unit sphere in $X$ respectively. $lin(X)$ denotes the linear span
of $X$ and $\overline{lin}(X)$ its closure. $ker(x^*)$ is the kernel of $x^*\in X^*$.
The density $dens(X)$ of a Banach space $X$ is  the minimal cardinality of a norm dense subset
of $X$. We write $X\equiv Y$ if Banach spaces $X$ and $Y$ are isometrically isomorphic.
   If $X$ is a Banach space and $I$ a set, then
$(x_i, x^*_i)_{i\in I}$ is called biorthogonal if $x_i^*(x_j)=\delta_{i, j}$ for any $i, j\in I$.
Such a biorthogonal system is called fundamental if $\{x_i: i\in I\}$ is linearly dense
in $X$ and it is called total if the span of $\{x_i: i\in I\}$ is dense in the weak$^*$ topology
in $X^*$. Moreover, it is called a Markushevich basis if it is both fundamental and total.
Recall that a Banach space $X$ is called injective if given any
pair of Banach spaces $Y\subseteq Z$ and any linear bounded operator
$T:Y\rightarrow X$ there is $S: Z\rightarrow X$ which extends $T$. Examples
of injective Banach spaces are spaces of the form $C(K)$ for $K$ extremally disconnected,
which are exactly the Stone spaces of complete Boolean algebras (see e.g. \cite{semadeni}).

For a compact Hausdorff space $K$ by $C(K)$ we mean the Banach
space of real-valued continuous functions on $K$ with the supremum norm.
For $x\in K$ an element $\delta_x\in C(K)^*$ is given by $\delta_x(f)=f(x)$ 
for all $f\in C(K)$. All topological spaces considered in the paper
are Hausdorff. $\chi(x, X)$ is the character of a point $x$ in the space $X$, i.e.,
the minimal cardinality of a neighbourhood base at $x$. The pseudocharacter of a point
in the space $X$ is the minimal cardinality of a family of open sets whose intersection is
$\{x\}$. It is well known that the pseudocharacter of a point in a  compact Hausdorff space
is equal to its character.
$Clop(K)$ stands for
the Boolean algebra of clopen subsets of a space $K$.

A hyperplane
is a one-codimensional subspace. By $L_p(\{0,1\}^\kappa)$ for
$p\in [1, \infty]$ and $\kappa$ a cardinal we mean $L_p(\mu)$, where
$\mu$ is the homogeneous probability product measure on $\{0,1\}^\kappa$.
The class of WLD (weakly Lindel\"of determined) Banach spaces has many nice characterizations,
the most convenient for this paper is the one as the class of Banach spaces $X$ which admit
a linearly dense set $D\subseteq X$  such that $\{d\in D: x^*(d)\not=0\}$
is countable for each $x^*\in X^*$ (\cite{gonzalez}). $X$ is a Grothendieck
Banach space if and only in  $X^*$ weakly$^*$ convergent sequences 
coincide with weakly convergent sequences.

 The continuum hypothesis abbreviated as
{\sf CH} denotes the statement `$\mathfrak c=\omega_1$".
The terminology concerning Martin's axiom,
dense sets, filters in partial orders and forcing can be found in \cite{kunen}.
Definitions of cardinal invariants like $\mathfrak p$, ${\mathfrak{add}}$, ${\mathfrak{cov}}$,  etc.,
and the information on the Cicho\'n and the van Douwen diagrams 
can be found in \cite{blass}. A subset $\mathbb A$ of a partial order $\PP$ is said to be centred
if for every finite $\mathbb B\subseteq \mathbb A$ there is $p\in \PP$ such that $p\leq q$ for every $q\in \mathbb B$.
A partial order $\PP$ is said to have precaliber $\omega_1$
if given an uncountable $\mathbb A\subseteq \PP$ there is a centred and uncountable $\mathbb B
\subseteq \mathbb A$.

\subsection{Some previous results}

The following two simple lemmas were implicitly used in \cite{russo}.

\begin{lemma}\label{dense-range} Suppose that $X$ and $Y$ are two Banach spaces
of the same density and  $T: X\rightarrow Y$ is a bounded linear operator whose range is dense in
$Y$.
If $Y$ does not admit an overcomplete set, then $X$ does not admit an overcomplete set.
\end{lemma}
\begin{proof} Let $\kappa$ be a cardinal such that the densities of $X$ and of $Y$ are
$\kappa$.
 Suppose that $D=\{d_\xi: \xi<\kappa\}$ is an overcomplete set in $X$.
Let $A\subseteq\kappa$ be of cardinality $\kappa$, and let $y\in Y$ and $\varepsilon>0$. There is
$x\in X$ such that $\|T(x)-y\|<\varepsilon/2$. Since $D=\{d_\xi: \xi<\kappa\}$ is overcomplete,
there is a finite linear combination $x'\in X$ of elements of $\{d_\xi: \xi\in A\}$
such that $\|x'-x\|<\varepsilon/2\|T\|$. So there is
a finite linear combination $y'=T(x')$ of elements of $\{T(d_\xi): \xi\in A\}$
satisfying $\|y'-y\|\leq \|y'-T(x)\|+\|T(x)-y\|<\varepsilon$. This shows that
every subset of $T[D]$ of cardinality $\kappa$ is dense in $Y$. Since the density of $Y$ is
$\kappa$ we conclude that $T[D]$ is overcomplete in $Y$.

\end{proof}

\begin{lemma}\label{unions} Suppose that $\lambda<cf(\kappa)$ are  uncountable cardinals and 
 $X$ is a Banach space of density $\kappa$
such that $X=\bigcup_{\xi<\lambda}X_\xi$, where $X_\xi$s are proper closed subspaces of
$X$. Then $X$ does not admit an overcomplete set.
\end{lemma}
\begin{proof}
Suppose that $D\subseteq X$ has cardinality $\kappa$. As 
$D=\bigcup_{\xi<\lambda}(D\cap X_\xi)$ and $\lambda<cf(\kappa)$, 
there is $\xi<\lambda$ such that $D\cap X_\xi\subseteq X_\xi$ has cardinality $\kappa$.
As $X_\xi$ is a proper closed subspace of $X$, the set $D$ is not overcomplete in $X$.
\end{proof}

\begin{theorem}[\cite{klee}]\label{klee} Suppose that $X$ is a Banach space, 
$\emptyset\not=B\subseteq \N$, 
$\{x_n: n\in B\}\subseteq X$ consists  of norm one vectors
and $\lambda_k$ for $k\in \N$ are distinct elements of the interval $(0,1/2)$. Let $y_k=\sum_{n\in B}\lambda_k^nx_n$ for each $k\in \N$.
Then for every infinite $C\subseteq \N$
we have $\overline{lin}(\{y_k: k\in C\})=\overline{lin}(\{x_n: n\in B\})$. Consequently every infinite dimensional separable
Banach space admits an overcomplete set.
\end{theorem}
\begin{proof} 
Let $x^*$ be a norm one linear bounded functional on $\overline{lin}(\{x_n: n\in B\})$.
It is enough to show that there is $k\in C$ such that $x^*(y_k)\not=0$.
 Define $\sigma_{n}=x^*(x_n)$ for $n\in B$ and
$\sigma_{n}=0$ for $n\in \N\setminus B$. We have 
$\limsup_{n\rightarrow \infty}\sqrt[n]{|\sigma_{n}|}
\leq\sup_{n\in B}\sqrt[n]{|x^*(x_n)|}\leq 1$ and so the formula
$$f(\lambda)=\sum_{n\in B}x^*(x_n)\lambda^n$$
defines an analytic function on $(-1, 1)$. This function is zero on $(-1, 1)$ only
if $x^*(x_n)=0$ for each $n\in B$, which is not the case since $x^*$ is not the zero functional
on $\overline{lin}(\{x_n: n\in B\})$. So $f$ cannot have infinitely many zeros in $(0, 1/2)$,
which means that  for some $k\in C$ we have 
$0\not=f(\lambda_k)=x^*(\sum_{n\in B}\lambda^n_kx_n)=x^*(y_k)$ as required.
\end{proof}

\begin{theorem}[\cite{russo}]\label{russo}
Suppose that $X$ is a Banach space.
\begin{enumerate} 
\item {\rm(}{\sf CH}{\rm)} If the density of $X^*$ is $\omega_1$, then
$X$ admits an overcomplete set.
\item {\rm(}$\neg${\sf CH}{\rm)} $\ell_\infty$ does not admit an overcomplete set. 
\item If $X$ admits a fundamental biorthogonal system and has density bigger than $\omega_1$,
then $X$ does not admit an overcomplete set.
\item If the cofinality of the density of $X$ is bigger than $\mathfrak c$,
then $X$ does not admit an overcomplete set.
\item $\ell_1(\omega_1)$ does not admit an overcomplete set.
\end{enumerate}
\end{theorem}

\subsection{General facts}

\begin{definition} Suppose that $X$ is a Banach space and $Y$ is its closed subspace.
For $y^*\in S_{Y^*}$ and $x\in X$ we define
\begin{enumerate}
\item $E(y^*)=\{x^*\in S_{X^*}: x^*|Y=y^*\}.$
\item $[y^*](x)=\{x^*(x): x^*\in E(y^*)\}$.
\end{enumerate}
\end{definition}

\begin{lemma}\label{character} Let $\kappa$ be an infinite cardinal and $X$ be  a Banach space
and $x^*\in S_{X^*}$. Then
$\chi(x^*, B_{X^*})\leq \kappa$ (with respect to the weak$^*$ topology) if and only if
there is a closed subspace $Y$ of $X$ of density $\leq \kappa$
such that $E(x^*|Y)=\{x^*\}$.
\end{lemma}
\begin{proof} If $\chi(x^*, B_{X^*})\leq \kappa$, then there is its open sub-basis of cardinality
not bigger than $\kappa$
consisting of sub-basic open sets of the form 
$U(y,\varepsilon)=\{y^*\in X^*: |(x^*-y^*)(y)|<\varepsilon\}$ for $y\in X$ and $\varepsilon>0$. 
If $Y$ is a subspace of $X$ generated by all such $y$s, it has density not bigger than
$\kappa$. Moreover if $y^*(y)=x^*(y)$ for any $y\in Y$, then $y^*=x^*$, that is
$E(x^*|Y)=\{x^*\}$

For the reverse implication let $D\subseteq Y$ be a norm dense set of cardinality not
bigger than $\kappa$ such that $E(x^*|Y)=\{x^*\}$. We claim that all finite intersections of the
sets of the form $U(d, \varepsilon)$ for $d\in D$ form a neighbourhood basis of $x^*$.
Since $B_{X^*}$ is compact in the weak$^*$ topology, the character of points is equal to their
pseudocharacter, that is, it is enough to prove that $x^*$ is the only point
of the intersection of such $U(d, \varepsilon)$ for $d\in D$. 
But if $y^*\in Y^*\setminus\{x^*|Y\}$, then there is $d\in D$ such that
$x^*(d)\not=y^*(d)$ and so there is $\varepsilon$ such that $y^*\not\in U(d,\varepsilon)$,
as required.
\end{proof}

\begin{lemma}\label{convex} Suppose that $X$ is a Banach space and $Y$ is its closed subspace and that 
 $y^*\in S_{Y^*}$ and $x\in X$. Then
$E(y^*)$ is a nonempty convex and closed in the weak$^*$ topology subset of $S_{X^*}$.
In particular, $[y^*](x)\subseteq \R$ is convex.
\end{lemma}
\begin{proof}
Note that $E(y^*)=\{x^*\in B_{X^*}: x^*|Y=y^*\}$ since 
for every $x^*\in E(y^*)$ already $Y$ contains witnesses for
$\|x^*\|\geq 1$.  It is clear that $(tx^*_1+(1-t)x^*_2)|Y=y^*$
for any $x_1^*, x_2^*\in E(y^*)$ and $0\leq t\leq 1$. 
Also if $x^*\in B_{X^*}\setminus E(y^*)$, then there is $y\in Y$
and $\varepsilon>0$ such that $|x^*(y)-y^*(y)|>\varepsilon$
and so $\{z^*\in B_{X^*}: z^*(y)\in (x^*(y)-\varepsilon, x^*(y)+\varepsilon)\}$
is a weak$^*$ open neighbourhood of $x^*$ disjoint form $E(y^*)$
which proves that $E(y^*)$ is closed. The nonemptyness follows from the Hahn-Banach theorem.
\end{proof}

\section{Positive results}

In this section we step-up the argument of Klee from \cite{klee} to some nonseparable Banach spaces
using a coherent sequence of injections from $\alpha<\omega_1$ into $\N$.  The
main applications of this are Theorem \ref{main-positive} and Corollary \ref{wld-iff}.

\begin{lemma}\label{ulam}
There is a sequence $(e_\alpha: \alpha<\omega_1)$ such that 
\begin{enumerate}
\item $e_\alpha: \alpha\rightarrow \N$ is injective for every $\alpha<\omega_1$,
\item $\{\beta<\alpha_1, \alpha_2: e_{\alpha_1}(\beta)\not=e_{\alpha_2}(\beta)\}$
 is finite
for every $\alpha_1, \alpha_2<\omega_1$.
\end{enumerate}
Consequently for every uncountable $A\subseteq \omega_1$ and every $\gamma<\omega_1$
there is an uncountable $A'\subseteq A\setminus \gamma$ such that 
$e_{\alpha_1}|(\gamma+1)=e_{\alpha_2}|(\gamma+1)$ for every $\alpha_1, \alpha_2\in A'$.
\end{lemma}
\begin{proof}
The construction of $(e_\alpha: \alpha<\omega_1)$ is by transfinite recursion and
is standard (see Ex. 28.1. of \cite{jech}).

To prove the second part of the lemma note that there is an uncountable $A_1\subseteq A\setminus\gamma$ and
a finite $F\subseteq \gamma+1$  such that for all $\alpha\in A_1$
$e_\alpha(\beta)=e_{\gamma+1}(\beta)$ for all $\beta\in \gamma\setminus F$. 
There is an uncountable $A'\subseteq A_1$ such that $e_{\alpha_1}|F=e_{\alpha_2}|F$
for all $\alpha_1, \alpha_2\in A'$. It follows that for all $\alpha\in A'$ we have the same
$e_\alpha|(\gamma+1)$. 
\end{proof}

\begin{theorem}\label{yes-zfc}
Suppose that $X$ is a Banach space   which admits a linearly dense set
$\{x_\alpha: \alpha<\omega_1\}$  such that there is a norm closed subspace 
$Y\subseteq X^*$ of finite codimension $n\in \N$
such that $\{\alpha<\omega_1: y^*(x_\alpha)\not=0\}$ is at most countable for each $y^*\in Y$.
Then there is $k\in \N$ with $k\leq n$ such that every subspace of $X$ of codimension $k\in\N$  admits an overcomplete set.
\end{theorem}
\begin{proof}

Let $(e_\alpha: \alpha<\omega_1)$ be as in Lemma \ref{ulam}
and let $B_\alpha\subseteq \N$ be the range of $e_\alpha$. Let $r_\alpha$ for
$\alpha<\omega_1$ be distinct elements of $(0,1/2)$. For $\alpha<\omega_1$ define $y_\alpha\in X$ by 
$$y_\alpha=\sum_{n\in B_\alpha}
      r_\alpha^{n} x_{e^{-1}_\alpha(n)}.$$ 

First we will prove that    whenever $y^*\in Y\setminus\{0\}$ and $A\subseteq \omega_1$ is uncountable
then there is $\alpha\in A$ such that $y^*(y_\alpha)\not=0$.       
 
 Let $\gamma<\omega_1$ be such that $y^*(x_\alpha)=0$ for $\gamma<\alpha<\omega_1$.
By Lemma \ref{ulam} there is an
uncountable $A'\subseteq A$ such that for all $\alpha\in A'$ we have the same (injective)
$e_\alpha|(\gamma+1)$. Let us call it $g:\gamma+1\rightarrow\N$. In particular $\gamma<\min(A')$.
By Theorem \ref{klee}
for $B\subseteq\N$ being the range of $g$ we obtain that
$$ \overline{lin}(\{z_\alpha: \alpha \in A'\})=\overline{lin}(\{x_\beta: \beta\leq\gamma\})\leqno(*)$$
where
$$z_\alpha= \sum_{n\in B}
      r_\alpha^{n} x_{g^{-1}(n)}.$$ 
In particular this means that there is $\alpha\in A'$ such that $y^*(z_\alpha)\not=0$.
But 
$$y^*(y_\alpha)=y^*(z_\alpha+ \sum_{n\in B_\alpha\setminus B}
       r_\alpha^{n} x_{e^{-1}_\alpha(n)})=y^*(z_\alpha)\not=0$$
       since       $e^{-1}_\alpha[B\setminus B_\alpha]=(\gamma, \alpha)\subseteq(\gamma, \omega_1)$.

Now we will show that the codimension of $\overline{lin}(\{y_\alpha: \alpha\in A\})$ is at most $n$ for
every uncountable $A\subseteq \omega_1$.     Otherwise there
are linearly independent  
$z_1^*, ..., z_{n+1}^*\in X^*$ such that  
$$\{y_\alpha: \alpha\in A\}\subseteq\bigcap\{ker(z_i^*): 1\leq i\leq n+1\}.$$
 We will derive
a contradiction from this hypothesis. Let $X^*=Y\oplus W$ where $W$ is $n$-dimensional.
Let $z_i^*=y_i^*+w_i^*$ where
$y_i^*\in Y$ and $w_i^*\in W$ for $1\leq i\leq n+1$. There is $(r_1, ..., r_{n+1})\in \R^{n+1}\setminus\{0\}$
such that $\sum_{1\leq i\leq n+1}r_iw_i^*=0$, so $\sum_{1\leq i\leq n+1}r_iz_i^*\in Y\setminus \{0\}$
as $z_1^*, ..., z_{n+1}^*$ are linearly independent. But this means that a nonzero element of $Y$ is zero on all the elements of 
$\{y_\alpha: \alpha\in A\}$ which contradicts our previous findings.
   
To conclude the theorem  we consider uncountable $A\subseteq \omega_1$ such that
$\overline{lin}(\{y_\alpha: \alpha\in A\})$ has
the biggest possible codimension in $X$. Then ${lin}(\{y_\alpha: \alpha\in A'\})$
is dense in $\overline{lin}(\{y_\alpha: \alpha\in A\})$ for any uncountable $A'\subseteq A$
and so $\{y_\alpha: \alpha\in A\}$ is overcomplete in $\overline{lin}(\{y_\alpha: \alpha\in A\})$
which is of some codimension $k\in \N$ for some $k\leq n$.
As all subspaces  of a fixed finite codimension of a Banach space are mutually
isomorphic this  shows that overcomplete sets are present in all 
subspaces of $X$ of codimension $k$.

\end{proof}

\begin{theorem}\label{main-positive} The following Banach spaces  admit overcomplete sets:
\begin{enumerate}
\item Every WLD Banach space of density $\omega_1$, in particular
\begin{enumerate}
\item $\ell_p(\omega_1)$, $L_p(\{0,1\}^{\omega_1})$ for $p\in (1,\infty)$, 
\item $L_1(\{0,1\}^{\omega_1})$,
\item $c_0(\omega_1)$.
\item $C(K)$s for $K$ a Corson compact where all Radon measure have separable supports.
\end{enumerate}
\item $C([0,\omega_1])$,
\item $C(K)$, where $K$ is the one point compactification of
a refinement of the order topology on $[0,\omega_1)$ obtained be isolating all points of some  
subset of $[0,\omega_1)$.
\end{enumerate}
\end{theorem}
\begin{proof} The proof will consist of showing that the above spaces satisfy
the hypothesis of Theorem \ref{yes-zfc}.

For (1) we apply Theorem \ref{yes-zfc} for $n=0$ as
a Banach space is WLD if and only if it admits a 
linearly dense set such that every functional is countably supported by
it (Theorem 7 of \cite{gonzalez}).

For (2) and (3) we apply Theorem \ref{yes-zfc} for $n=1$. We identify
the compactification point with $\{\omega_1\}$.
The dual spaces to the spaces from (2) and (3) are $\ell_1([0, \omega_1])$
as the spaces are scattered (\cite{rudin}). As $Y\subseteq C(K)^*$
we consider 
$$Y=\{\mu\in \ell_1([0,\omega_1]): \mu(\{\omega_1\})=0\}.$$ 
As the linearly dense set we consider
$$D=\{1_{\{\alpha\}}: \alpha \ \hbox{\rm  is isolated in}\ K\}\cup
\{1_{[0, \omega_1]}\}\cup\{1_{(\alpha,\omega_1]}: \alpha<\omega_1\}.$$
It is clear that any $\mu\in Y$ is zero on all but countably
many elements of $D$. Also $D$ is linearly dense as
$1_{(\alpha, \beta]}=1_{(\alpha,\omega_1]}-1_{(\beta,\omega_1]}$
for any $\alpha<\beta<\omega_1$ and $1_{[0,\alpha]}=1_{\{0\}}+1_{(0,\alpha]}$.
Moreover  all clopen sets of $K$ are finite unions of intervals and
characteristic functions of clopen sets generate $C(K)$ as $K$ is totally disconnected
since it is scattered and compact.

So Theorem \ref{yes-zfc} implies that either $C(K)$ admits
an overcomplete set or hyperplanes of $C(K)$ admit overcomplete sets.
But hyperplanes of such $C(K)$ are isomorphic to the entire $C(K)$ since
$K$ admits nontrivial convergent sequences as it is a scattered compact space.
\end{proof}

\begin{corollary}\label{wld-iff} A WLD Banach space $X$ admits an overcomplete set if and only if
the density of $X$ is less or equal to $\omega_1$.
\end{corollary}
\begin{proof} 
The existence follows from Theorem \ref{main-positive} (1) and 
the nonexistence from the results of \cite{russo} (in this paper Theorem \ref{russo} (2)).
\end{proof}

\section{Consistency results}

The purpose of this section is to prove the consistency negative results 
in Theorems \ref{main-consistency} and \ref{nonmonolithic} and in Corollary
\ref{cor-psi}. The general case of a nonseparable Banach space $X$ of density smaller than $\mathfrak c$
with its dual ball $B_{X^*}$ separable or even nonmonolithic
in the weak$^*$ topology can be reduced to the case of a Banach space of the form
$C(K)$ for $K$ separable with a countable dense set $D$
and nonmetrizable of weight $\kappa<\mathfrak c$ in Theorems \ref{main-consistency}
and \ref{nonmonolithic} and Proposition \ref{cover-hyper}. 
To exclude the existence of overcomplete sets in such Banach spaces we prove that certain sets
can be covered by countably many hyperplanes and any linearly dense set contains
such a subset of cardinality $\kappa$ (Proposition  \ref{ma-partition}). Having this,
if $cf(\kappa)>\omega$, then one hyperplane contains $\kappa$ many elements of the set,
so the set cannot be overcomplete. In the proof of Proposition \ref{ma-partition}
we use Martin's axiom {\sf MA} and the negation of {\sf CH} to build
linear bounded functionals  $y\in l_1(D)$ whose kernels are the covering hyperplanes 
mentioned above. For this, in Definition \ref{def-P}, we  define a partial order $\PP$
of finite approximations to such $y$s and prove density lemmas and the countable
chain condition of $\PP$ in Lemmas \ref{xi-density}, \ref{hard-density} and  \ref{ccc}.
In fact, the proof of the c.c.c. of $\PP$ is relatively hard and the rest of the Section is
related to the necessity of such a complicated $\PP$. In Lemma \ref{precaliber} we prove
that simpler partial orders (having precaliber $\omega_1$) cannot do the job of $\PP$.
It follows 
in Theorem \ref{ma-precaliber}  that weaker versions of Martin's axiom are
consistent with the existence of overcomplete sets in Banach spaces like in 
Theorems \ref{main-consistency} and \ref{nonmonolithic} and that the results obtained
in \cite{russo} under {\sf CH} are consistent with arbitrary size of the continuum (Theorem \ref{any-c}).
Understanding forcing is required only when reading the proofs of the latter results i.e.,
Lemma \ref{precaliber} and Theorems \ref{ma-precaliber} and \ref{any-c}. For the first part of this section
we assume from the reader some familiarity with Martin's axiom and the way one applies it,
 for this the reader may consult \cite{jech, kunen}.

\begin{definition}\label{def-P} Let $K$ be a
compact Hausdorff space with a dense subset $\{d_n: n\in \N\}$ and $\kappa$ an infinite cardinal.
Let $\{x_\xi: \xi<\kappa\}\subseteq K$ be distinct nonisolated 
points and $\{f_\xi: \xi<\kappa\}\subseteq C(K)$
satisfy $f_\xi(x_\xi)=1$, $f_{\xi}(x_\eta)=0$ for all $\xi<\eta<\kappa$ and
$\|f_\xi\|\leq M$ for all $\xi<\kappa$ and a rational $M>2$.

We define a partial order $\PP$ consisting of
conditions $p=(n_p, y_p, X_p, \varepsilon_p)$ such that

\begin{enumerate}[(a)]
\item $\varepsilon_p\in \Q_+$, $n_p\in \N$, $n_p>0$,
\item $y_p: n_p\rightarrow  \Q$,  $y_p(0)\not=0$,
\item $X_p$ is a finite subset of $\kappa$,
\item $1-\sum_{n<n_p} |y_p(n)|=\delta_p\geq \varepsilon_p2M^{3|X_p|+1}$,
\item $|\sum_{n<n_p} y_p(n)f_\xi(d_n)|<
{\varepsilon_p}$ for every $\xi\in X_p$.
\end{enumerate}
We declare $p\leq q$ if
\begin{enumerate}[(i)]
\item $\varepsilon_p\leq \varepsilon_q$,
\item $y_p\supseteq y_q$,
\item $n_p\geq n_q$,
\item $X_p\supseteq X_q$.
\end{enumerate}
\end{definition}
It is easy to see that $\PP$ with $\leq$ is a partial order.
\begin{lemma}\label{xi-density} Let $\PP$ and the corresponding objects be as in Definition \ref{def-P}.
For every $\xi<\kappa$ there is $p\in \PP$ such that
$\xi\in X_p$. 
\end{lemma}
\begin{proof} We define $p=(n_p, y_p, X_p, \varepsilon_p)$ by
putting $n_p=1$, $X_p=\{\xi\}$. 
The value of $y_p(0)$ is chosen so that
 $0<y_p(0)\leq 1/2$ and  $|y_p(0)f_\xi(d_0)|< 1/(4M^4)$ hold, $\varepsilon_p=1/(4M^4)$, 
$\delta_p=\varepsilon_p2M^{4}=1/2\leq 1-|y_p(0)|$.

\end{proof}

 The density of the set $E$ in the following Lemma will be crucial in the proof of
 the c.c.c. of $\PP$ in Lemma \ref{ccc} and is the most technical part of this section.
 
\begin{lemma}\label{hard-density} Let $\PP$ and the corresponding objects be as in Definition \ref{def-P}. Let $n, k\in \N$. The following sets are dense in $\PP$.
\begin{itemize} 
\item $C_n=\{p\in \PP: n_p\geq n\}$,
\item $D_k=\{p\in \PP: \varepsilon_p\leq 1/k\}$,
\item $E=\{p\in \PP: \delta_p\geq \varepsilon_p2M^{6|X_p|+1}\}$,
\end{itemize}
\end{lemma}
\begin{proof} The density of $C_n$s is clear as 
$(n, y, X_p, \varepsilon_p)\leq p$ for any $n\geq n_p$ and
$y: n\rightarrow \Q$ such that $y|[n_p, n)=0$.
So given  $k\in \N$ and $q\in \PP$ let us focus on finding $p\leq q$ in $D_k\cap E$ which will
finish the proof of the lemma.
Let $X_q=\{\xi_0, ..., \xi_{m-1}\}$ for $\xi_0< ... <\xi_{m-1}<\omega_1$ and some $m\in \N$. 
First we will find $y: m\rightarrow \R$ such that
\begin{enumerate}
\item 
$|\sum_{n<n_q} y_q(n)f_{\xi_j}(d_n)+\sum_{i<m} y(i)f_{\xi_j}(x_{\xi_i})|=0$ for every $j<m$. 
\item $|y(j)|\leq \varepsilon_q+M\sum_{i<j}|y(i)|$ 
for every $j< m$,
\end{enumerate}
We do it by induction on $0\leq j<m$.
 Suppose that we
 are done for $i< j<m$. Define
 $$y(j)=-\sum_{n<n_q} y_q(n)f_{\xi_{j}}(d_n)-\sum_{i< j} y(i)f_{{\xi_{j}}}(x_{\xi_i}).$$
As $f_{x_{\xi_{j}}}(x_{\xi_{j}})= 1$ and 
$f_{x_{\xi_{j}}}(x_{\xi_{i}})=0$ for $j<i<m$ we obtain (1).
Note that we keep (2) as 
 $|y(j)|\leq\varepsilon_q+M\sum_{i< j}|y(i)|$ since $\|f_\xi\|\leq M$ for
 every $\xi<\kappa$ and by Definition \ref{def-P} (e). 
 Now we note that 
  \begin{enumerate}
  \item[(3)] $|y(j)|\leq \varepsilon_qM^{3j}$ for all $j<m$, 
\item[(4)] $\sum_{i<m} |y(i)|\leq \delta_q/2$.
\end{enumerate}  
To prove (3)  by induction on $j<m$ we
use (2) and the fact that $\sum_{i<j}M^i<M^j$ for every $j\in \N$ since $M>2$: 
$$|y(j)|\leq \varepsilon_q+M\sum_{i<j}|y(i)|\leq \varepsilon_q+\varepsilon_qM\sum_{i<j}M^{3i}\leq$$
$$\leq\varepsilon_q(1+MM^{3(j-1)+1})=\varepsilon_q(1+M^{3j-1})\leq\varepsilon_qM^{3j}.$$
So $\sum_{i<m} |y(i)|\leq \varepsilon_qM^{3m+1}\leq \delta_q/2$ by Definition \ref{def-P} (d)
which gives (4) and completes the proof of the properties of $y$.

Now we are ready to start defining $p\leq q$ such that $p\in D_k\cap E$. Let
$\theta\in \Q_+$ satisfy the following:
\begin{enumerate}
\item[(5)] $m\theta\leq \delta_q/4$,
\item[(6)] $m\theta(\theta+M+1)2M^{6m+1}\leq \delta_q/4$
\item[(7)] $m\theta(\theta+M+1)\leq \min( 1/k, \varepsilon_q)$.
\end{enumerate}
Let $y_i\in \Q$ for $i<m$ be such that $|y_i-y(i)|<\theta$ 
for every $i<m$ and let $n_i\in \N$  for $i<m$ be distinct and such that $n_i>n_q$ and
$|f_{\xi_{j}}(x_{\xi_{i}})- f_{\xi_{j}}(d_{n_i})|<\theta$ for every $i, j<m$. This can be achieved
because $\Q$ is dense in $\R$ and $\{d_n: n\in\N\}$ is dense in $K$,
 $f_{\xi_j}$s are continuous, and $x_\xi$s are nonisolated. Define $p=(n_p, y_p, X_p, \varepsilon_p)$ as follows: $n_p=\max(\{n_i: i<m\})+1$, for $n<n_p$ define
$$
y_p(n) =
  \begin{cases}
    y_q(n) & \text{if $n<n_q$,} \\
    y_i & \text{if $n=n_i$, $i<m$} \\
    0 & \text{otherwise for $n<n_p$}.
  \end{cases}
 $$
 $X_p=X_q$, $\varepsilon_p=m\theta(\theta+M+1)$. First let us check that $p\in \PP$.
 Condition (a) - (c) of Definition \ref{def-P} are clear. To prove condition (d) note
 that by (4) and by the choice of $y_i$ we have   $\sum_{i<m} |y_i|\leq \delta_q/2+m\theta$,
 and so using  (5) and (6) we conclude that 
 $$\delta_p=\delta_q-\sum_{i<m}|y_i|\geq \delta_q/2-m\theta\geq \delta_q/4
 \geq \varepsilon_p2M^{6m+1}= \varepsilon_p2M^{6|X_p|+1}\leqno (8)$$
as required in Definition \ref{def-P} (d). To prove (e) note
that by (1) and by  the choice of $n_i$s and the fact that $|y(i)|\leq 1$ for
each $i<m$  (which follows (4) and \ref{def-P} (d) for $q$) we conclude that
$$|\sum_{n<n_p} y_p(n)f_{\xi_j}(d_n)|=
|\sum_{n<n_q} y_p(n)f_{\xi_j}(d_n)+\sum_{i<m} y_p(n_i)f_{\xi_j}(d_{n_i})|=$$
$$=|\sum_{i<m} y_p(n_i)f_{\xi_j}(d_{n_i})-\sum_{i<m} y(i)f_{\xi_j}(x_{\xi_i})|\leq $$
$$\leq M\sum_{i<m} |y_p(n_i)-y(i)| +\sum_{i<m} |f_{\xi_j}(d_{n_i})-f_{\xi_j}(x_{\xi_i})|\leq$$
$$m\theta(M+1)\leq
 m\theta(\theta+M+1)=
\varepsilon_p$$
 for every $j<m$ which is condition (e) of Definition \ref{def-P}.
 We have $p\leq q$ by (7). Also $p\in D_k$ by (7). Finally $p\in E$ by (8).

\end{proof}

\begin{lemma}\label{ccc} $\PP$ satisfies c.c.c. 
\end{lemma}
\begin{proof} Given 
$p_\xi=(n_{p_\xi}, y_{p_\xi}, X_{p_\xi}, \varepsilon_{p_\xi})\in \PP$ for $\xi<\omega_1$
by Lemma \ref{hard-density} we may assume that $p_\xi\in E$ for each $\xi<\omega_1$
and by passing to an uncountable set we may assume that 
$n_{p_\xi}=n$, $y_{p_\xi}=y$, $\varepsilon_{p_\xi}=\varepsilon$
for some $n\in \N$, $y: n\rightarrow \Q$ and $\varepsilon\in \Q_+$.
We claim that then $(n, y, X_{p_\xi}\cup X_{p_\eta}, \varepsilon)\leq p_\xi, p_\eta$.
The only nonclear part of Definition \ref{def-P} to check is (d), but it follows from
the fact that the conditions are in $E$ of Lemma \ref{hard-density}.
\end{proof}

\begin{lemma}\label{ma-partition}{\rm(}{\sf MA+$\neg$CH}{\rm)}
Let $K$ be a compact Hausdorff space with a dense subset $\{d_i: i\in \N\}$ and $\kappa$
an uncountable cardinal satisfying $\kappa<\mathfrak c$.
Suppose that  $\{x_\xi: \xi<\kappa\}\subseteq K$ are
distinct nonisolated points and $\{f_\xi: \xi<\kappa\}\subseteq C(K)$
satisfy $f_\xi(x_\xi)=1$, $f_{\xi}(x_\eta)=0$ for all $\xi<\eta<\kappa$ and
$\|f_\xi\|\leq M$ for all $\xi<\kappa$ and some $M>0$.
Then there are sets $B_m\subseteq \kappa$ for $m\in\N$ such that $\bigcup_{m\in \N}B_m=\kappa$
and $y_m\in \ell_1\setminus\{0\}$ for $m\in \N$ such that
$$\sum_{i\in \N} y_m(i)f_\xi(d_i)=0$$
for all $\xi\in B_m$ and all $m\in \N$.
\end{lemma}
\begin{proof} If $M\leq 2$, we replace it with some $M>2$. In any case we may assume that it is a rational.
Let $\PP$ be the partial order from Definition \ref{def-P}.
We consider the countable power $\mathbb S$ with finite supports of partial order $\PP$ with
coordinatewise order. 
By  Lemma \ref{ccc} and
{\sf MA}+$\neg${\sf CH} we know that finite products of
$\PP$ satisfy the  c.c.c. and so $\mathbb S$ satisfies the c.c.c. Applying {\sf MA} let
$G\subseteq \mathbb S$ be a filter in $\mathbb S$ meeting the following dense sets
for $\xi\in\kappa$ and $k, n, m\in \N$:
$$F_\xi=\{s\in \mathbb S: \exists k\in \N\ \xi\in X_{s(k)}\}$$
$$C_{n, m}=\{s\in \mathbb S: s(m)\in C_n\}$$
$$D_{k, m}=\{s\in \mathbb S: s(m)\in D_k\}$$
The density of these sets follows from Lemmas \ref{xi-density} and \ref{hard-density}
and the fact that the supports of the conditions of the product are finite.
In particular if $s\in \mathbb S$ and $\xi\in\kappa$ we find $k\in \N$ not belonging
to the support of $s$ and define $s'\leq s$ with $s'\in F_\xi$ using 
Lemma \ref{xi-density} on the coordinate $k$.

For $m\in \N$ define  $B_m=\bigcup\{X_{s(m)}: s\in G\}$.
By the density of each $F_\xi$ for each $\xi<\kappa$ we have $\bigcup_{m\in \N}B_m=\kappa$.
Let $y_m=\bigcup\{y_{s(m)}: s\in G\}$. It follows from the conditions (b) and (d)
of Definition \ref{def-P} and the
density of the sets $C_{n, m}$ for $n, m\in \N$
that $y_m\in \ell_1\setminus\{0\}$ for each $m\in \N$. The final condition
of the lemma follows from the density of the sets $C_{n, m}$ and $D_{k, m}$ for $k, n,  m\in \N$ and the condition
(e) of Definition \ref{def-P}.
\end{proof}

\begin{proposition}\label{cover-hyper}{\rm(}{\sf MA+$\neg$CH}{\rm)}
 Suppose that $X$ is a Banach space whose dual unit  ball $B_{X^*}$ is
separable in the weak$^*$ topology. Let $\kappa<\mathfrak c$ be
a cardinal and $\{x_\xi: \xi<\kappa\}\subseteq X$
be a  set satisfying $x_\xi\not\in \overline{lin}\{x_\eta: \eta<\xi\}$ for every $\xi<\kappa$. Then
$\{x_\xi: \xi<\kappa\}$ can be covered by countably many hyperplanes of $X$. 
\end{proposition}
\begin{proof} Let $\{d_i^*: i\in \N\}$ be a countable set dense in $B_{X^*}$ with
 the weak$^*$ topology. We may assume that  $\kappa$ is uncountable, in particular that
 $X^*$ is nonseparable in the norm. First we will argue that we may assume that
 $\sum_{i\in \N}\alpha_id_i^*=0$ if and only if $\alpha_i=0$ for every $i\in \N$.
 
For this let $Y$ be the norm closure of $\{d_i^*: i\in \N\}$ in $X^*$.
As $X^*$ is nonseparable and $Y$ is separable the quotient $X^*/Y$ is infinite dimensional
and so there is an infinite biorthogonal system
 $(e_i, \phi_i)_{i\in \N}$ in $X^*/Y\times (X^*/Y)^*$. Let
$\pi: X^*\rightarrow X^*/Y$ be the quotient map.  Let $\psi_i=\phi_i\circ \pi\in X^*$
and let $e_i^*\in X^*$ be such that 
$\pi(e^*_i)=e_i$. 

Now let us prove that $\{c_i^*: i\in \N\}$
has the desired properties, where $c_i=(1-1/(i+1))d_i^*+ e^*_i/(i+1)\|e^*_i\|$ for each $i\in \N$ 
i.e., that $\{c_i^*: i\in \N\}$ is  dense in $B_{X^*}$ with
 the weak$^*$ topology and 
 $\sum_{i\in \N}\alpha_ic_i^*=0$ if and only if $\alpha_i=0$ for every $i\in \N$.
 The density is clear. Also $\psi_j(\sum_{i\in \N}\alpha_ic_i^*)= \alpha_j/(j+1)\|e^*_j\|$
 so we obtain the second property as well and so
 may assume  that
 $\sum_{i\in \N}\alpha_id_i^*=0$ if and only if $\alpha_i=0$ for every $i\in \N$.

For every $\xi<\kappa$ there is a norm one functional $x_\xi^*\in X^*$
such that $x_\xi^*$ is zero on $\overline{lin}\{x_\eta: \eta<\xi\}$
and $x_\xi^*(x_\xi)\not=0$. By multiplying $x_\xi$s we may assume that
$x_\xi^*(x_\xi)=1$ for each $\xi<\kappa$. We can divide $\kappa$ into countably
many sets $A_n\subseteq \kappa$ such that
each $\{x_\xi: \xi\in A_n\}$ is norm bounded. Now consider $K=B_{X^*}$ with
 the weak$^*$ topology. For $\eta<\kappa$ define continuous functions $f_\eta: K\rightarrow\R$
 by $f_\eta(x^*)=x^*(x_\eta)$ and note that $f_\eta(x^*_\xi)=0$ if $\eta<\xi<\kappa$
 and $f_\xi(x_\xi)=1$. It follows from Lemma \ref{ma-partition} 
 that for each $n\in \N$ we can find $B^m_n\subseteq A_n$ for $m\in \N$
 such that $\bigcup_{m\in \N} B^m_n=A_n$ for each $n\in \N$ and 
  $y^m_n\in\ell_1\setminus\{0\}$
 satisfying for each $\xi\in B^m_n$
 $$(\sum_{i\in \N}y^m_n(i)d_i^*)(x_\xi)=\sum_{i\in \N}y^m_n(i)f_\xi(d_i^*)=0.$$
 
 Note that  $\sum_{i\in \N}y^m_n(i)d_i^*\in X^*\setminus\{0\}$ by the property
 of $\{d_i: i\in\N\}$ from the first part of the proof and by the fact that 
 $y^m_n\in\ell_1\setminus\{0\}$. Hence
 $$H^m_n=\{x\in X: (\sum_{i\in \N}y^m_n(i)d_i^*)(x)=0\}$$ 
 is a hyperplane of $X$. So
 each set $\{x_\xi: \xi\in B^m_n\}$ is included in
 a hyperplane, 
 as required.
\end{proof}

\begin{theorem}\label{main-consistency}{\rm(}{\sf MA+$\neg$CH}{\rm)}
 Suppose that the density of a Banach space $X$ is smaller than $\mathfrak c$ and has uncountable
 cofinality and that the  dual unit  ball $B_{X^*}$ is
separable in the weak$^*$ topology. Then $X$ does not admit an overcomplete set.
\end{theorem}
\begin{proof} Suppose that $D$ is a linearly dense subset of $X$. We will show
that it is not overcomplete.  As the density of $X$ is $\kappa$ we can 
 construct
 $\{x_\xi: \xi<\kappa\}\subseteq D$
satisfying $x_\xi\not\in \overline{lin}\{x_\eta: \eta<\xi\}$ for every $\xi<\kappa$. Then
$\{x_\xi: \xi<\kappa\}$ can be covered by countably many hyperplanes of $X$ by Proposition
\ref{cover-hyper}.  Since the cofinality of $\kappa$ is uncountable, one of these
hyperplanes contains $\kappa$ many vectors $x_\xi$ which shows that $D$ is not overcomplete.
\end{proof}

Recall that a topological space is called monolithic if the closures of countable sets are metrizable.

\begin{theorem}\label{nonmonolithic}{\rm(}{\sf MA+$\neg$CH}{\rm)} Suppose that
$X$ is a Banach space of density $\omega_1$ whose dual ball is not monolithic in the
weak$^*$ topology. Then $X$ does not admit an overcomplete set. 
\end{theorem}
\begin{proof} 
 Let  $D=\{d_i: i\in \N\}\subseteq B_{X^*}$ be a countable set whose weak$^*$-closure in $X^*$
is nonmetrizable. We will construct a linear bounded operator $T: X\rightarrow Y$ whose
range is dense in $Y$ and $Y$ is a Banach space of density $\omega_1$ whose dual ball
is weak$^*$ separable. This will be enough by Lemma \ref{dense-range} and Theorem \ref{main-consistency}.

Define a  linear bounded operator $T:X\rightarrow\ell_\infty$
by $T(x)=(d_i(x))_{i\in\N}$. Let $Y$ be the norm closure of $T[X]$ in $\ell_\infty$.
Let $\delta_i\in (\ell_\infty)^*$ be defined by $\delta_i(f)=f(i)$ for $f\in \ell_\infty$ and $i\in \N$.
By the Krein-Millman theorem the rational convex combinations
of $\delta_i$s are weak$^*$ dense in the dual ball of $\ell_\infty$. In particular
the dual ball $B_{{\ell_\infty}^*}$ is separable in the weak$^*$ topology.
The dual ball
of $Y$  with the weak$^*$ topology is  a continuous image of
$B_{{\ell_\infty}^*}$ (taking restrictions
of functionals is a continuous map which is onto by the Hahn-Banach theorem). 
So $B_{Y^*}$ is separable in the weak$^*$ topology as required.

To show that $Y$ is nonseparable in the norm it is enough to show
that $B_{Y^*}$ is nonmetrizable in the weak$^*$ topology. 
The dual ball $B_{Y^*}$ maps continuously by $T^*$ onto
$T^*[B_{Y^*}]\subseteq X^*$ which contains $\{T^*(\delta_i): i\in \N\}=\{d_i: i\in \N\}$.
As $T^*[B_{Y^*}]$ is compact, it contains the closure of $\{d_i: i\in \N\}$
which is nonmetrizable, so $B_{Y^*}$ is nonmetrizable since continuous images of
compact metrizable spaces are metrizable.
\end{proof}

\begin{corollary}\label{cor-psi}{\rm(}{\sf MA+$\neg$CH}{\rm)} Suppose that $\A$ is
an almost disjoint family of subsets of $\N$ of cardinality $\kappa<\mathfrak c$ of
uncountable cofinality.
Then the Banach space generated in
$\ell_\infty$ by $c_0$ and $\{1_A: A\in \A\}$ does  not admit an overcomplete set. 
\end{corollary}
\begin{proof}
As a nonseparable subspace of $\ell_\infty$ the space satisfies the hypothesis
of Theorem \ref{main-consistency}.
\end{proof}

It is well-know that the space above is isometric to $C_0(K_\A)$ where
$K_\A$ is locally compact scattered space of weight $\kappa$ and of Cantor-Bendixson height two
known as $\Psi$-space, Mr\'owka-Isbell space or Alexandroff-Urysohn space.
One can see that the  dual of the space above  has density $\kappa$ as well.

The remaining part of this section is devoted to results showing that the positive 
CH results of \cite{russo} are consistent with any size of the continuum.
The first result, Theorem \ref{ma-precaliber}, also   shows that the 
a relatively complex Definition \ref{def-P} and a relatively delicate argument 
in Lemma \ref{hard-density} are unavoidable.
\begin{lemma}\label{precaliber} Suppose that $X$ is a Banach space of density $\omega_1$ which
admits an overcomplete set and that $\PP$ is
a partial order which has precaliber $\omega_1$. Then $\PP$ forces that
the completion of $X$ admits an overcomplete set.
\end{lemma}
\begin{proof}
Let $D=\{x_\alpha: \alpha<\omega_1\}$ be an overcomplete set in $X$. Let $\dot X$ stand
for a $\PP$-name for the  completion of $\check X$ in the generic extension by $\PP$. We claim
that $\PP$ forces that $\check D$ is  an overcomplete set in $\dot X$.  

Let $\dot A$ be a $\PP$-name  for an uncountable  subset of $\omega_1$, $\varepsilon>0$  and $\dot x$ be a 
$\PP$-name for an element of the completion of $X$
and let $\{\dot \alpha_\xi: \xi<\omega_1\}$ be $\PP$-names such that
$\PP\forces \dot A=\{\dot \alpha_\xi: \xi<\omega_1\}$. By the density of $X$ in its
completion we can find $p\in \PP$ and $x\in X$ such that  $p\forces\|\check x-\dot x\|<\varepsilon/2$.
For each $\xi<\omega_1$ find $p_\xi\leq p$  and $\alpha_\xi\in \omega_1$ such that
$p_\xi\forces \dot \alpha_\xi=\check\alpha_\xi$. 

Since $\PP$ has precaliber $\omega_1$, there is an uncountable $B\subseteq\omega_1$ such 
that any finite subset of $\{p_{\alpha_\xi}: \xi\in B\}$  has a  lower bound in $\PP$.

 Since $D$ is overcomplete in $X$ we have
$\xi_1, ...,\xi_k\in B$ and $r_i\in \R$  for $1\leq i\leq k$ for some $k\in \N$ such that 
$\|x-\sum_{1\leq i\leq k}r_ix_{\alpha_{\xi_i}}\|<\varepsilon/2$. Then
$$p\forces \|\dot x-\sum_{1\leq i\leq k}r_i\check x_{\alpha_{\xi_i}}\|<\varepsilon.$$
Let
$q\leq p_{\alpha_{\xi_1}}, ..., p_{\alpha_{\xi_k}}$. Then
$$q\forces \{\check{\alpha_{\xi_1}}, ..., \check{\alpha_{\xi_k}}  \}\subseteq\dot A.$$
This shows that $q$ forces that the distance of $\dot x$ from
the closure of the linear span of $\{x_{\alpha}: \alpha \in \dot A\}$ is smaller than $\varepsilon$.
Since $\varepsilon $ was arbitrary it shows that 
$\PP$ forces that $\{x_{\alpha}: \alpha \in \dot A\}$ is linearly dense in the 
completion of $X$. Since $\dot A$ was an arbitrary $\PP$-name for an uncountable subset of
$\omega_1$ this proves that $D$ remains an overcomplete set in the completion of $X$.

\end{proof}

\begin{theorem}\label{ma-precaliber} It is consistent with {\sf MA} for partial orders having precaliber $\omega_1$
and the negation of {\sf CH} that every Banach spaces whose dual has density $\omega_1$
admits an overcomplete set.
\end{theorem}
\begin{proof} Let $V$ be a model of ZFC and {\sf GCH}. Let 
$(\PP_\alpha, \dot \Q_\alpha)_{\alpha\leq \omega_2}$ be a finite support iteration
of forcings of cardinality $\omega_1$, having precaliber $\omega_1$ such that $V[G_{\omega_2}]$ satisfies
$2^\omega=\omega_2$ and 
Martin's axiom for partial orders having precaliber $\omega_1$ and where $G_{\omega_2}$
is  $\PP_{\omega_2}$-generic over $V$. Let $G_\alpha=G_{\omega_2}\cap \PP_{\alpha}$ for
any $\alpha\leq\omega_2$.

Let $X$ be any Banach space in $V[G_{\omega_2}]$ whose dual has density $\omega_1$. Let $E\subseteq X$
be a dense linear (non-closed) subspace over $\Q$ of $X$ of cardinality $\omega_1$ . Without loss of generality we may assume that $E=\omega_1$. So some functions  
$+: \omega_1\times \omega_1\rightarrow \omega_1$ and 
$\cdot: \Q\times \omega_1\rightarrow \omega_1$ represent linear operations in $E$ and 
$\|\ \|: \omega_1\rightarrow \R$ represents  the norm on $E$. So $X$ is the completion of $E$
in $V[G_{\omega_2}]$. Using the c.c.c. of $\PP_{\omega_2}$ and applying the standard arguments we can find
$\alpha<\omega_2$  such that 
$\PP_{\omega_2}$ forces 
that $+, \cdot, \|\ \|$ are in $V[G_\alpha]$. As $\PP_\alpha$
is a finite support iteration of c.c.c. forcings of cardinality $\omega_1$ 
and $\alpha<\omega_2$ the model $V[G_\alpha]$ satisfies ${\sf CH}$.
It follows that the completion $X_\alpha$ of $E$ in $V[G_\alpha]$ admits an 
overcomplete set $D\subseteq X_\alpha$
because the dual $X_\alpha^*$ in $V[G_\alpha]$ must have density $\omega_1$, as otherwise, 
a norm discrete subset $\{\phi_\alpha: \alpha<\omega_2\}\subseteq X_\alpha^*$  by the Hahn-Banach theorem
would produce a norm discrete subset of $X^*$ of cardinality $\omega_2$ in 
$V[G_{\omega_2}]$ contradicting
the choice of $X$. 

By the standard argument (see e.g., 1.5.A of  \cite{tomek}) the iteration
$\PP_{\omega_2}$ is equivalent to the iteration $\PP_\alpha* \dot{\mathbb{S}}^\alpha$ 
where $\PP_\alpha$ forces that $\dot{\mathbb{S}}^\alpha$ is a finite
support iteration of forcings having precaliber $\omega_1$. But such an
iteration has precaliber $\omega_1$ (e.g., Theorem 1.5.13 of  \cite{tomek}).
So we are in the position to apply Lemma \ref{precaliber} in $V[G_\alpha]$
to conclude that $D$ stays  overcomplete in $X$ in $V[G_{\omega_2}]$.
\end{proof}

\begin{theorem}\label{any-c} The statement  that every Banach space whose dual has density $\omega_1$
admits an overcomplete set is consistent with any size of the continuum.
\end{theorem}
\begin{proof} Let $V$ be a model of ZFC which satisfies {\sf GCH} and let $\kappa$ be any cardinal
of uncountable cofinality and for an infinite $A\subseteq \kappa$ let $\PP_A$ stands for 
the partial order for adding 
Cohen reals labelled by elements of $A$, 
that is $\PP_A$ consist of finite partial functions from $A$ into $\{0,1\}$ and
is considered with the inverse inclusion as the order.  
Let $G_\kappa\subseteq \PP_\kappa$ 
be  $\PP_\kappa$-generic
over $V$. Let $G_A=G_\kappa\cap \PP_A$. As is well know (\cite{kunen}, \cite{jech}) 
the continuum of the model $V[G_\kappa]$ assumes value $\kappa$.
We will show that any Banach space in $V[G_\kappa]$ whose dual has density $\omega_1$ admits in
$V[G]$ an overcomplete set.

Let $X$ be any Banach space in $V[G_\kappa]$ whose dual has density $\omega_1$. Let $E\subseteq X$
be a dense linear (non-closed) subspace over $\Q$ of $X$ of cardinality $\omega_1$ . Without loss of generality we may assume that $E=\omega_1$. So some functions  
$+: \omega_1\times \omega_1\rightarrow \omega_1$ and 
$\cdot: \Q\times \omega_1\rightarrow \omega_1$ represent linear operations in $E$ and 
$\|\ \|: \omega_1\rightarrow \R$ represents  the norm on $E$. So $X$ is the completion of $E$
in $V[G_\kappa]$. Using the c.c.c. of $\PP_\kappa$ and applying the standard arguments we can find
$A\subseteq \kappa$ in $V$ of cardinality $\omega_1$ such that 
$\PP_\kappa$ forces that $+, \cdot, \|\ \|$ are in $V[G_A]$.
 As $\PP_A$ adds only $\omega_1$ Cohen reals due to
the fact that $A$ has cardinality $\omega_1$ the model $V[G_A]$ satisfies ${\sf CH}$.
It follows that the completion $X_A$ of $E$ in $V[G_A]$ admits an overcomplete set $D\subseteq X_A$
because the dual $X_A^*$ in $V[G_A]$ must have density $\omega_1$, as otherwise, 
a norm discrete subset $\{\phi_\alpha: \alpha<\omega_2\}\subseteq X_A^*$  by the Hahn-Banach theorem
would produce a norm discrete subset of $X^*$ of cardinality $\omega_2$ in $V[G]$ contradicting
the choice of $X$. 

By the standard argument $\PP_\kappa$ is isomorphic with $\PP_A\times \PP_{\kappa\setminus A}$ and
so by the product lemma $V[G]=V[G_A][G_{\kappa\setminus A}]$. Since $\PP_{\kappa\setminus A}$
has precaliber $\omega_1$ in $V[G_A]$ we are in the position to apply Lemma \ref{precaliber}
to conclude that $D$ stays  overcomplete in $X$ in $V[G]$.

\end{proof}

\section{Negative results}

Recall the definitions of $E(y^*)$ and 
$[y^*](x)$ from Section 2.

\begin{lemma}\label{interval} Suppose that $\kappa$ is  an infinite cardinal and that 
$X$ is a Banach space of density $\kappa$, its subspace $Y\subseteq X$ has density smaller than $\kappa$
and $D\subseteq X$ is linearly dense in $X$. Let $y^*\in S_{Y^*}$ be such that
$\chi(x^*, B_{X^*})=\kappa$ for every $x^*\in E(y^*)$. Then, for every subspace
$W$ with $Y\subseteq W\subseteq X$, $dens(W)<\kappa$ and every $w^*\in S_{W^*}$
with $w^*|Y=y^*$,  there is $d\in D$ such that $[w^*](d)$ contains a nondegenerate interval in $\R$.
\end{lemma}
\begin{proof} By the hypothesis and Lemma \ref{convex}  and Lemma \ref{character} the set
$E(w^*)$ is a convex closed subset of $S_{X^*}$ which contains at least two
distinct points $x_1^*, x_2^*$. The set $\{x\in X: x_1^*(x)=x_2^*(x)\}$ is a closed
proper subspace of $X$ and hence there is $d\in D$ which does not belong to it,
i.e., without loss of generality we have  $x_1^*(d)<x_2^*(d)$. So 
$$(x_1^*(d), x_2^*(d))\subseteq \{(tx_1^*+(1-t)x_2^*)(d): 0\leq t\leq 1\}\subseteq [w^*](d).$$
\end{proof}

\begin{lemma}\label{single-value} Suppose that $\kappa$ is a cardinal
of uncountable cofinality, $X$ is a Banach space of density $\kappa$,
$Y$ is a subspace of $X$ of density smaller than $\kappa$ and $y^*\in S_{Y^*}$ is such that
$\chi(x^*, B_{X^*})=\kappa$ for all points $x^*\in E(y^*)$.
Suppose that $D\subseteq X$ is of cardinality $\kappa$ and  such that 
$D\setminus E\subseteq X$ is
 linearly dense in $X$ for every $E\subseteq D$ of cardinality less than $\kappa$.
 
  Then the set
of all $x^*\in S_{X^*}$ for which there is $D'\subseteq D$ of cardinality $\kappa$
 such that the set $\{x^*(d): d\in D'\}$ is a singleton  is weakly$^*$ dense in $E(y^*)$.
In particular, $X$ does not admit an overcomplete set.
\end{lemma}
\begin{proof} 
Let us first conclude the last part of the lemma from the main part.  Suppose 
that $D\subseteq X$ is overcomplete. So using the main part of the lemma 
find an $x^*_1\in E(y^*)$ and $D'\subseteq D$ of cardinality $\kappa$ such that
$\{x^*_1(d): d\in D'\}=\{r\}$ for an $r\in \R$. If $r=0$ we conclude that
$D'$ is a subset of a hyperplane which contradicts the hypothesis that $D$ is overcomplete.
If $r\not=0$, use the fact that $D'$ is linearly dense, since $D$ is overcomplete
and use again the main part of the lemma finding $D''\subseteq D'$ of cardinality $\kappa$
and a $x_2^*\in E(y^*)$ such that $\{x_2^*(d): d\in D''\}=\{s\}$ for an $s\in \R$. 
The functional $x_2^*$ can be taken different from $\pm x^*_1$ (using the fact that
$E(y^*)$ cannot contain both of the  $x^*_1$ and $-x^*_1$ as they both cannot extend $y^*$
and using the fact that $E(y^*)$ is not a singleton by Lemma \ref{character})
 and so linearly independent from $x^*_1$.
So for $z^*={s\over r}x^*_1-x_2^*\not=0$ we have $D''\subseteq\{x\in X: z^*(x)=0\}$ which
contradicts the hypothesis that $D$ is overcomplete.

So now, let us turn to the proof of the main part of the lemma.
Let $D=(d_\alpha: \alpha<\kappa)$  be an enumeration of $D$. Let 
$U=\{x^*\in X^*: x^*(x_i)\in I_i, 1\leq i\leq k\}\cap E(y^*)$ be a
nonempty  weakly$^*$ open subset of $E(y^*)$ 
where
$x_1, ..., x_k\in X$, $k\in \N$ and $I_i$s are nonempty open intervals in $\R$.

First let us prove that there is a  closed subspace $Y\subseteq W\subseteq X$  
of density less than $\kappa$ 
with  $x_1, ..., x_k\in W$ and
a functional $w^*\in W^*$ of norm one satisfying $w^*|Y=y^*$ and $w^*(x_i)\in I_i$
for all $ 1\leq i\leq k$ and there is a nondegenerate open interval $I\subseteq \R$
such that 
\begin{enumerate}
\item for every closed  $Z$ of density smaller than $\kappa$ satisfying $W\subseteq Z\subseteq X$ and 
\item for every $z^*\in S_{Z^*}$ satisfying $z^*|W=w^*$ and 
\item for every $A\subseteq \kappa$ satisfying $|A|<\kappa$
\end{enumerate}
there is $\beta\in\kappa\setminus A$ such that
$$I\subseteq [z^*](d_\beta).$$
Indeed, if this was not the case, then for
every  closed subspace $Y\subseteq W\subseteq X$ of density smaller than $\kappa$ 
 such that $x_1, ..., x_k\in W$ 
and every norm one functional
$w^*\in S_{W^*}$  satisfying $w^*|Y=y^*$ and $w^*(x_i)\in I_i$
for all $ 1\leq i\leq k$ 
and every nondegenerate interval $I$ with rational endpoints there is
a closed  $Z$ of density smaller than $\kappa$ satisfying $W\subseteq Z\subseteq X$ 
and  $z^*\in S_{Z^*}$ satisfying $z^*|W=w^*$ 
and $A\subseteq \kappa$ satisfying $|A|<\kappa$ such that for every $\beta\in\kappa\setminus A$  we have 
$I\not\subseteq [z^*](d_\beta)$. 

Let $Y_1$ be the subspace of $X$ generated by $Y$ and $x_1, ...,  x_k$ and let $y_1^*\in E(y^*)\cap U$.
Enumerating all nondegenerate intervals  with rational endpoints as $(J_n)_{n\in \N}$
we could recursively construct increasing sequence $(W_n)_{n\in \N}$
of  closed subspaces of $X$  of densities smaller than $\kappa$  and increasing subsets 
 $(A_n)_{n\in \N}$ of $\kappa$ of cardinalities smaller than $\kappa$ and  $(w_n^*)_{n\in \N}$ satisfying
 $W_0\supseteq Y_1$, 
$w_0|Y_1=y^*_1|Y_1$ and 
$w_n^*\in S_{W_n^*}$ and $w_{n+1}^*|W_n=w_n^*$ for every $n\in\N$ and
$w^*_n(x_i)\in I_i$
for all $ 1\leq i\leq k$ 
and  $J_n\not\subseteq [w_n^*](d_\beta)$ for every $\beta\in \kappa\setminus A_n$.
Take $W$ to be the closure of $\bigcup_{n\in \N}W_n$ and
$w^*\in S_{W^*}$ to be the unique functional satisfying $w^*|W_n=y_n$ for each $n\in \N$ and
put  $A=\bigcup_{n\in \N}A_n$. $W$ has density smaller than $\kappa$ and $A$
has cardinality smaller than $\kappa$
by the uncountable cofinality of $\kappa$. So $\{d_\xi: \xi\in \kappa\setminus A\}$
is linearly dense by the hypothesis of the lemma and  by Lemma \ref{interval}
there is $n\in \N$ such that $J_n\subseteq [w^*](d_\beta)\subseteq [w^*_n](d_\beta)$ for 
some $\beta\in\kappa\setminus A\subseteq \kappa\setminus A_n$.
But this contradicts the choice of $w^*_n$ and completes the proof of the existence
 of $W, w^*, I$  as in (1) - (3).

So let $W, w^*, I$ be as in (1) - (3). Let
$r\in I$.  Now by transfinite recursion we
can construct an increasing sequence $(Z_\xi)_{\xi<\kappa}$ of closed 
subspaces of $X$ and a sequence $(z_\xi^*)_{\xi<\kappa}$  and
a sequence $(\alpha_\xi)_{\xi<\kappa}$ of distinct elements of $\kappa$ such that
\begin{itemize}
\item $Z_0=W$, $z_0^*=w^*$,
\item $Z_\xi$ has density not bigger than the maximum of the density of $W$ and 
the cardinality of $\xi$,
\item $z_\xi^*\in S_{Z_\xi^*}$,
\item $z_\xi^*|Z_\eta=z^*_\eta$ for every $\eta<\xi<\kappa$,
\item $z_{\xi+1}^*(d_{\alpha_\xi})=r$.
\end{itemize}
Given $Z_\xi$, $z_\xi^*$  and  $\{\alpha_\eta: \eta< \xi\}$ as above, use (1) - (3) to find
 $\alpha_{\xi}\in \kappa\setminus\{\alpha_\eta: \eta<\xi\}$
 and $Z\supseteq Z_\xi$ and $z^*\in S_{Z^*}$ such that
$z^*(d_{\alpha_{\xi}})=r\in I$. Now define 
$Z_{\xi+1}$ as the subspace of $X$ generated by
$Z_{\xi}$ and $d_{\alpha_{\xi}}$ and $z_{\xi+1}^*\in S_{Z^*_{\xi+1}}$ such that
$z_{\xi+1}^*=z^*|   Z_{\xi+1}$.  
Then we also have $z_{\xi+1}^*(d_{\alpha_{\xi}})=r$. At a limit stage $\lambda<\kappa$
define $Z_\lambda=\overline{\bigcup_{\xi<\lambda}Z_\xi}$ and
$z_\lambda^*$ to be a norm one extension of $\bigcup_{\xi<\lambda}z_\xi^*$ to $Z_\lambda$.

Let $Z=\bigcup_{\xi<\kappa} Z_\xi$ and $z^*\in Z^*$ be such that
$z^*|Z_\xi=z_\xi^*$ for every $\xi<\kappa$. By extending $z^*$ to $X$ we have
$x^*\in X^*$ such that $x^*(d_{\alpha_\xi})=r$ for
all $\xi<\kappa$, moreover such an $x^*$ is in $U\cap E(y^*)$ since $x^*|Y_1=w^*|Y_1=y_1^*|Y_1$ as required. 

\end{proof}

As a corollary we obtain the following:

\begin{proposition}\label{cor-single-value} Let $X$ be a Banach space of density $\kappa$ of uncountable cofinality.
If $\chi(x^*, B_{X^*})=\kappa$ for every $x^*\in S_{X^*}$, then $X$ does not admit an overcomplete set.
\end{proposition}

\begin{lemma}\label{lemma-completion}  Suppose that $\kappa$ is an uncountable cardinal. There is an injective
Banach space $X_\kappa$ and its subspace $Y_\kappa$ of density $\kappa$   such that
$\chi(z^*, B_{Z^*})\geq\kappa$ for any $z^*\in S_{Z^*}$ and any subspace $Z$ satisfying $Y_\kappa\subseteq Z\subseteq X_\kappa$.
\end{lemma}
\begin{proof} Consider the Boolean algebra $\B$ of all clopen subsets of
$\{0,1\}^\kappa$ with the product topology and the Boolean completion $\A$ of $\B$, that is a 
complete Boolean algebra, where $\B$ is dense, i.e.,  such that for every $a\in \A\setminus\{0\}$
there is $b\in \B\setminus\{0\}$ satisfying $b\leq a$, where $\leq$ is the Boolean order.
$\B$ is generated as a Boolean algebra by sets $b_\alpha=\{x\in \{0,1\}^\kappa: x(\alpha)=1\}$
for $\alpha<\kappa$. By the density of $\B$ in $\A$
every element  $a\in \A$ is the supremum of a maximal pairwise disjoint
set of elements of $\B$. As $\B$ is c.c.c. its pairwise disjoint sets of elements are
at most countable. 
 For every $A\subseteq \kappa$ consider the Boolean algebra $\B_A$ generated
by $\{b_\alpha: \alpha\in A\}$ and the Boolean algebra
$\A_A\subseteq \A$ of all
Boolean suprema in $\A$ of all sets of elements of $\B_A$. It is easy to check that
$\A_A$ is a (complete) subalgebra of $\A$ for every $A\subseteq \kappa$.
So by the previous observation we have
$$\A=\bigcup\{\A_A: A\subseteq\kappa, \ |A|\leq\omega\}.$$

Let $K$ be the Stone space of $\A$, i.e., a compact Hausdorff space whose algebra
of clopen sets is isomorphic to $\A$ (we will identify it with $\A$). As $\A$ is complete,
$K$ is extremally disconnected and in particular it is totally disconnected and $C(K)$
is an injective Banach space. Put $X_\kappa=C(K)$. As $K$ is totally disconnected,
$$X_\kappa=C(K)=\overline{lin}(\{1_U: U \ \hbox{clopen in} \ K\})=\overline{lin}(\{1_a: a\in \A\}).$$
We can define some subspaces of $X_\kappa$:
$$Y_\kappa=\overline{lin}(\{1_a: a\in \B\}),$$
$$X_A=\overline{lin}(\{1_a: a\in \A_A\}).$$
As the cardinality of $\B$ is $\kappa$, it is clear that the density of $Y_\kappa$ is $\kappa$
as required.
Again, since $K$ is totally disconnected, every element of $C(K)$ can be approximated by a sequence
of linear combinations of characteristic functions of clopen sets, so
$$X_\kappa=\bigcup\{X_A: A\subseteq\kappa, \ |A|\leq\omega\}.$$
and consequently every subspace $W\subseteq X_\kappa$ of density less than $\kappa$ is included
in a subspace of the form $X_A\subseteq X_\kappa$ where $|A|<\kappa$.

Now fix a subspace $Z\subseteq X_\kappa$ such that $Y_\kappa\subseteq Z$ and let us prove
that $\chi(z^*, B_{Z^*})\geq \kappa$ for every $z^*\in S_{Z^*}$. Using  Lemma \ref{character}
it is enough to prove that given a subspace $W\subseteq Z$ of density less than $\kappa$
and $w^*\in S_{W^*}$ there are $x^*_1, x^*_2\in S_{X_\kappa^*}$ such that $x_1^*|W=w^*=x_2^*|W$
and there is $\alpha\in \kappa$ such that $x^*_1(1_{b_\alpha})\not=x_2^*(1_{b_\alpha})$ 
(since $b_\alpha\in Y_\kappa\subseteq Z$).

Let $A\subseteq \kappa$ be such that $|A|<\kappa$ and $W\subseteq X_A$. By the Hahn-Banach theorem
extend $w^*$ to an element $v^*$ of $S_{X^*_A}$. Let $\alpha$ be any element of
$\kappa\setminus A$.
Note that $X_A$ is isometric to $C(K_A)$, where $K_A$ is the Stone space
of the Boolean algebra $\A_A$ as it is generated by the characteristic functions 
of elements of $\A_A$. Consider the Boolean subalgebra $\mathcal C$ of $\A$
generated by $\A_A\cup\{b_\alpha\}$. 
Elements of $\A_A$  are suprema of elements of $\B_A$  and $\alpha\not\in A$, so
for each nonzero element $a$ of $\A_A$ we have 
$a\cap b_\alpha\not=\emptyset\not= a\setminus b_\alpha$. It follows that 
each ultrafilter of $A_\A$ (i.e., a point of $K_A$) has exactly two
extensions to ultrafilters in $\mathcal C$, one containing $b_\alpha$ and one containing its
complement. It follows that the Stone space $L$ of $\mathcal C$ is homeomorphic to
$K_A\times\{0,1\}$. We will identify it with  $K_A\times\{0,1\}$. Note that
under this identification $b_\alpha=K_A\times\{1\}$  and $L\setminus b_\alpha=K_A\times\{0\}$
and the embedding $\iota$ of $C(K_A)$ into $C(L)$ is given by
$\iota(f)=f\circ \pi$, where $\pi: K_A\times \{0,1\}\rightarrow K_A$  is the canonical projection.

By the Riesz representation theorem let $\mu$ be a Radon
measure on $K_A$ corresponding to $v^*$. 
Consider
two  measures 
$\nu_0, \nu_1$ on $L=K_A\times \{0,1\}$ given by
$$\nu_i(V_0\times\{0\}\cup V_1\times\{1\})=\mu(V_i)$$ 
for Borel subsets $V_0, V_1$ of $K_A$ and $i=0,1$. 
They satisfy $\mu(f)=\nu_0(f\circ \pi)=\nu_1(f\circ \pi)$ for
$f\in C(K_A)$, $\nu_i|C(K_A)=\nu_i|X_A=v^*$ for $i=0,1$. 
Also the variation norm of each $\nu_i$ is the same as for $\mu$, hence it is  one.
Moreover $\nu_i(b_\alpha)=i$.
So by the Hahn-Banach theorem there are two norm one extensions $x_1^*, x_2^*\in S_{X^*}$ 
of $\nu_0$ and $\nu_1$ respectively. They are extensions of $v^*$ 
satisfying  $x^*_1(1_{b_\alpha})\not=x_2^*(1_{b_\alpha})$
which completes the proof
of the required property of $v^*$ and the proof of the lemma.

\end{proof}

\begin{theorem}\label{main-negative} Let $\kappa$ be a cardinal of uncountable cofinality.
If $X$ is a Banach space of density $\kappa$ which 
contains an isomorphic copy of  $\ell_1(\kappa)$, then
$X$ does not admit an overcomplete set. Consequently 
the following Banach spaces do not admit overcomplete sets:
\begin{enumerate}
\item $C(K)$ for any infinite extremally disconnected  compact Hausdorff $K$.
\item $\ell_\infty(\lambda)$, $\ell_\infty(\lambda)/c_0(\lambda)$,
$L_\infty(\{0,1\}^\lambda)$ for any infinite cardinal $\lambda$. 
\item $C([0,1]^\kappa)$, $C(\{0,1\}^\kappa)$.
\end{enumerate}
\end{theorem}

\begin{proof} Let $Y_\kappa\subseteq X_\kappa$ be as in Lemma \ref{lemma-completion}.
 The universal property of $\ell_1(\kappa)$ implies that that there is
 a surjective bounded linear operator $T:\ell_1(\kappa)\rightarrow Y_\kappa$.
 By the injectivity of $X_\kappa$
  there is a bounded linear extension $R: X\rightarrow X_\kappa$ of $T$.
 Applying Lemma \ref{lemma-completion} for $\overline{R[X]}$ we conclude
 that $\chi(x^*, B_{\overline{R[X]}^*})=\kappa$ for every 
 $x^*\in S_{\overline{R[X]}^*}$.
 So Proposition \ref{cor-single-value}  implies that $\overline{R[X]}$ does not admit an overcomplete set.
 Now Lemma \ref{dense-range} yields that $X$ does not admit an overcomplete set.
 To conclude the second part of the theorem we will note that the Banach spaces in question
 contain appropriate nonseparable versions of $\ell_1$ and will use the first part of the Theorem.
 
 (1) $K$ is extremally disconnected if and only if the Boolean algebra
 $Clop(K)$ of clopen subsets of $K$ is complete. By Balcar-Franek theorem (\cite{balcar-franek})
 $Clop(K)$ contains an independent family $\F$ of cardinality equal to $|Clop(K)|$.
 $\{1_A-1_{K\setminus A}: A\in \F\}$ generates a copy of $\ell_1(|Clop(K)|)$ in
 $C(K)$. As $K$ is totally disconnected, we have $|Clop(K)|=dens(C(K))$, so
 we have $\ell_1(dens(C(K)))\subseteq C(K)$. To use the first part of the theorem it is now
 enough to note that $cf(dens(C(K))>\omega$. This is because
 $|Clop(K)|^\omega=|Clop(K)|$ by a theorem of Pierce (\cite{pierce}) and
 $cf(\kappa^\omega)>\omega$ for any cardinal $\kappa$ by the K\"onig Theorem (5.13 of \cite{jech}).
 
 (2)
The spaces $\ell_\infty(\lambda)$  are isomorphic to
the spaces $C(\beta\lambda)$ respectively and $\beta\lambda$ is extremally disconnected,
so apply (1). The spaces $L_\infty(\{0,1\}^\lambda)$  are isomorphic to
the spaces $C(HY_\lambda)$ respectively,  where $HY_\lambda$ is the Hewitt-Yosida space, i.e. 
the Stone space of the homogeneous measure algebra of Maharam type $\lambda$. $HY_\lambda$ is extremally disconnected,
so apply (1).

To prove the nonexistence of overcomplete sets in the spaces
$X=\ell_\infty(\lambda)/c_0(\lambda)$ we note that the quotient map
is an isometry on the  copy of $\ell_1(\kappa)$ for $\kappa=2^\lambda$
of the form  $\{1_A-1_{K\setminus A}: A\in \F\}$ from the proof of (1).
This is because the intersections in infinite independent families must be infinite
and the only characteristic functions of clopen sets which are in $c_0(\lambda)$ are
characteristic functions of finite sets.

(3)
The coordinate functions in  $C(\{-1,1\}^\kappa)$ or $C([-1,1]^\kappa)$ 
 generate  a copy of $\ell_1(\kappa)$
and obviously these spaces are isometric to $C(\{0,1\}^\kappa)$ or $C([0,1]^\kappa)$
respectively.
\end{proof}
 
A nice characterization of Banach spaces containing $\ell_1(\kappa)$
for $\kappa$ of uncountable cofinality can be found in \cite{talagrand2}. 

We observe that the nonexistence of overcomplete sets in  all Banach spaces
$X$ which contain $\ell_1(dens(X))$ for
$dens(X)\geq\omega_2$ can be directly
concluded from Theorem 3.6 of of \cite{russo} (In this paper Theorem \ref{russo} (3)),
Lemma \ref{pel-ros} below, Lemma \ref{dense-range} and Theorem 2 of \cite{davis} which says that
a nonseparable Banach space has a fundamental biorthogonal system if it has a weakly compactly generated quotient of the same density.
This argument covers
all Banach spaces considered in Theorem \ref{main-negative} of densities
bigger or equal to $\omega_2$. Theorem \ref{main-negative} provides new results
for densities $\omega_1$ but also provides a uniform ZFC proof for spaces
like $\ell_\infty$, $\ell_\infty/c_0$, $L_\infty(\{0,1\}^{\omega_1})$
whose densities may be equal to $\omega_1$ or may be bigger than $\omega_1$ depending on 
{\sf CH}.

\begin{lemma}[\cite{pel, ros}]\label{pel-ros} Let $\kappa$ be an infinite cardinal. If $X$ is a Banach space
containing an isomorphic copy of $\ell_1(\kappa)$, then $X$ admits $\ell_2(\kappa)$ as its quotient.
\end{lemma}
\begin{proof} If a Banach space $X$ contains an isomorphic  copy of
$\ell_1(\kappa)$, then $X^*$ contains an isomorphic copy of $(C(\{0,1\}^\kappa))^*$ by
Proposition 3.3. of \cite{pel}, hence $X^*$ contains
an isomorphic copy of $L_1(\{0,1\}^\kappa)$. Now, following Proposition 1.5 of \cite{ros}, we prove that 
$L_1(\{0,1\}^\kappa)$ contains an isomorphic copy of $\ell_2(\kappa)$. For
this we need to note that $\{0,1\}^\kappa$ is a compact abelian group and the product measure is
its Haar measure and that $\{\gamma_\xi: \xi<\kappa\}$ forms a set of independent characters
in the sense of Definition 1.4 of \cite{ros}, where $\gamma_\xi(x)=x(\xi)$ for
$x\in \{0,1\}^\kappa$ and $\xi<\kappa$. Let $T: \ell_2(\kappa)\rightarrow X^*$ be the obtained  isomorphic embedding.
Let $J: X\rightarrow X^{**}$ be the canonical embedding. 
So $T^*\circ J: X\rightarrow (\ell_2(\kappa))^*\equiv\ell_2(\kappa)$. To complete the proof
it is enough to show that $T^*\circ J$
is surjective. For this it is enough to show that $(T^*\circ J)^*$ is injective.
But as $T^{**}=T$ since $\ell_2(\kappa)$ is reflexive, we have 
 that $(T^*\circ J)^*=J^*\circ T^{**}=J^*\circ T=T$ which is injective.
\end{proof}

\begin{corollary}\label{unconditional} 
Let $X$ be a Banach space of density $\omega_1$ with an unconditional basis.
$X$ admits an overcomplete set if and only if $X$ is WLD.
\end{corollary}
\begin{proof} If $X$ is WLD and of density $\omega_1$, then $X$ admits an overcomplete set by
Theorem \ref{main-positive}. 
By Theorem 1.7 of \cite{argyros-rocky} a Banach spaces with an unconditional basis
is WLD if and only if $\ell_1(\omega_1)$ does not isomorphically embed into $X$. So if
$X$ is not WLD we have a copy of $\ell_1(\omega_1)$ in $X$ and may conclude that
$X$ does not admit an overcomplete set using Theorem \ref{main-negative}.
\end{proof}

\begin{corollary}\label{argyros-bigger} Suppose that $\kappa>\omega_1$
is a cardinal of uncountable cofinality and that 
 $X$ is a Banach space of density $\kappa$ whose dual contains an isomorphic
copy of $L_1(\{0,1\}^\kappa)$. Then $X$ does not admit an overcomplete set.
\end{corollary}
\begin{proof}
By  Argyros' solution of Pe\l czy\'nski's conjecture (\cite{argyros-tams}) if $\kappa>\omega_1$ and
the dual of a Banach space $X$ contains $L_1(\{0,1\}^\kappa)$, then $X$
contains $\ell_1(\kappa)$. Now apply Theorem \ref{main-negative}.
\end{proof}

\begin{corollary}\label{argyros-omega1} Whether every Banach space of
density $\omega_1$ whose dual contains
 $L_1(\{0,1\}^{\omega_1})$ admits an overcomplete set is undecidable.
\end{corollary}
\begin{proof} Let $X$ be a Banach space of density $\omega_1$.
By  Argyros' solution of Pe\l czy\'nski's conjecture (\cite{argyros-tams}) 
under {\sf MA}+$\neg${\sf CH} if 
the dual of a Banach space $X$ contains $L_1(\{0,1\}^{\omega_1})$, then $X$
contains $\ell_1(\omega_1)$. So  applying Theorem \ref{main-negative} one concludes that
$X$ does not admit an overcomplete set.
On the other hand Haydon's example from \cite{haydon-ch} 
constructed under {\sf CH}  is a Banach space
of the form $C(K)$ whose dual contains $L_1(\{0,1\}^{\omega_1})$ and
the densities of $C(K)$ and $C(K)^*$ are $\omega_1$. It follows from
the main result of \cite{russo}  that $C(K)$ admits an overcomplete set (Theorem \ref{russo}).
\end{proof}

\begin{corollary}\label{argyros-ma} {\rm (}$\mathfrak p=\mathfrak c>\omega_1${\rm )} No nonreflexive Grothendieck space of regular density (in particular equal to $\mathfrak c$)
 admits an overcomplete set.
\end{corollary}
\begin{proof} Assume $\mathfrak p=\mathfrak c>\omega_1$. 
The cardinal $\mathfrak p$ is a regular cardinal (Theorem 3.1. of \cite{van-douwen}).
It is proved in \cite{hlo}
that under the assumption $\mathfrak p=\mathfrak c>\omega_1$ 
every nonreflexive Grothendieck space has $\ell_\infty$
as a quotient (in fact, it is concluded from
the existence of an isomorphic copy of  $\ell_1(\mathfrak c)$ in the space).
If the density of $X$ is $\mathfrak c$, then the above result and Lemma \ref{dense-range}
and Theorem \ref{main-negative} imply that $X$ does not admit an
overcomplete set. 
If the density of $X$ is regular and bigger than $\mathfrak c$, then the 
statement of the corollary follows from results of \cite{russo} (which is Theorem \ref{russo} (4)
in this paper).
\end{proof}

In fact, Theorem  \ref{bigger-groth} excludes in ZFC densities of any cofinality bigger than $\omega_1$.
So the role of the hypothesis $\mathfrak p=\mathfrak c>\omega_1$ above is limited to
excluding the possibility of the existence of a nonreflexive Grothendieck space of
density $\omega_1$ (Note that nonreflexive Grothendieck spaces of density $\omega_1<\mathfrak{c}$
consistently exist (\cite{rogerio}). See also the discussion after Question \ref{q-groth}).

The following lemma will be used in the proof of Theorem \ref{ck-groth}.

\begin{lemma}\label{two-characters} Let $\kappa$ be an infinite cardinal
and $K$ be a Hausdorff compact space. Let $P(K)$ denote the
set of all Radon probability measures on $K$. Then $\chi(x^*, B_{{C(K)}^*})\geq\kappa$ for
all $x^*\in S_{{C(K)}^*}$ if and only if $\chi(\mu, P(K))\geq\kappa$ for every $\mu\in  P(K)$.
\end{lemma}
\begin{proof} We will interpret Radon measures on $K$ as functionals on $C(K)$.
Also note that $P(K)$  is a closed, and so compact subset of
$B_{C(K)^*}$. This is because if $\mu\in B_{C(K)^*}\setminus P(K)$,
then $\mu(1_K)<1$ or $\mu(f)<0$ for some $f\geq0$ and $f\in C(K)$
and both of these conditions define weak$^*$ open sets. It follows that
the pseudocharacter of points in $P(K)$ relative to $P(K)$ is equal to
their character relative to $P(K)$.

The forward implication follows from the fact that $P(K)$ is a 
$G_\delta$ subset of $B_{C(K)^*}$ being equal to
the intersection of all sets $\{\mu\in B_{C(K)^*}: \mu(1_K)>1-1/n\}$ for $n\in \N$.
Hence $\mu\in P(K)$ satisfying $\chi(\mu, P(K))<\kappa$ would also
have a pseudobasis in $B_{C(K)^*}$ of cardinality smaller than $\kappa$.

Now we will focus on proving the backward implication. Assume
that $\chi(\mu, P(K))\geq\kappa$ for every $\mu\in  P(K)$. Suppose that $\mu\in B_{C(K)^*}$.
Any intersection of less than $\kappa$-many weakly$^*$ open sets containing $\mu$ includes
 an intersection of the form
$$\bigcap_{x\in Y}\bigcap_{n\in \N}\{\nu\in X^*: |\nu(x)-\mu(x)|<1/n\}=
\bigcap_{x\in Y}\{\nu\in X^*: \nu(x)=\mu(x)\},$$
where $Y\subseteq X$ is of cardinality less than $\kappa$. So to prove that
$\chi(\mu, B_{C(K)^*})\geq\kappa$, which is equivalent to proving that
the pseudocharacter of $\mu$ relative to $B_{C(K)^*}$ is bigger or equal to $\kappa$,
it is enough to prove that the intersections of the above form always contain
some $\nu\in B_{C(K)^*}\setminus\{\mu\}$.

To do so  fix $Y\subseteq X$ of cardinality less than $\kappa$ 
and decompose $\mu$ as $\mu=\mu_+-\mu_-$, where
$\mu_+, \mu_-$ are positive Radon measures on $K$ with disjoint supports. 
By extending $Y$ we may assume that $1_K\in Y$.
Note
that $\mu_+/\|\mu_+\|, \mu_-/\|\mu_-\|\in P(K)$.
As the pseudocharacter of $\mu_+/\|\mu_+\|$ relative to $P(K)$ is not smaller than $\kappa$,
there is $\nu\in P(K)\setminus\{\mu_+/\|\mu_+\|\}$  which is in
$$\bigcap_{x\in Y}\bigcap_{n\in \N}\{\nu\in X^*: |\nu(x)-(\mu_+/\|\mu_+\|)(x)|<1/n\}=
\bigcap_{x\in Y}\{\nu\in X^*: \|\mu_+\|\nu(x)=\mu_+(x)\}.$$
In particular
$$\|\mu_+\|\nu(1_K)+\mu_-(1_K)=\mu_+(1_K)+\mu_-(1_K)=\mu(1_K)\leq 1$$
and $\|\mu_+\|\nu+\mu_-$ is positive as the sum of two positive measures, so 
$(\|\mu_+\|\nu-\mu_-)\in B_{C(K)^*}$. Moreover
$(\|\mu_+\|\nu-\mu_-)(x)=\mu(x)$ for all $x\in Y$ and 
$\|\mu_+\|\nu-\mu_-\not=\mu$ as $\nu\not=\mu_+/\|\mu_+\|$, as required.
\end{proof}

\begin{theorem}\label{ck-groth} Suppose that $K$ is an infinite compact Hausdorff space such that
$C(K)$ is Grothendieck space of density $\omega_1$. Then $C(K)$ does not admit
an overcomplete set.
\end{theorem}
\begin{proof} 
As is well know, if $C(K)$ is Grothendieck, the $K$ has no nontrivial convergent sequence and
so $K$ is not scattered, in particular there is a perfect $L\subseteq K$.
As an infinite closed subset of $K$, it must be nonmetrizable, again by the nonexistence
of nontrivial convergent sequences. So $C(L)$ is a quotient of $C(K)$ 
of density $\omega_1$ and is Grothendieck as this property is preserved by taking quotients.
We will prove that $C(L)$ does not admit an overcomplete set, which is enough by
Lemma \ref{dense-range}.

Theorem 3.5 and Proposition 5.3. of \cite{krupski} imply that if $L$ has no isolated points and
$C(L)$ is Grothendieck, then no probability Radon measure on $L$ is a $G_\delta$ point
in the space $P(L)$ of all probability Radon measures  on $L$. Lemma \ref{two-characters}
implies that $\chi(x^*, B_{{C(L)}^*})=\omega_1$ for every $x^*\in B_{{C(L)}^*}$ and
Proposition \ref{cor-single-value} implies that $C(L)$ does not admit an overcomplete set, as required.
\end{proof}

Note that it is possible (consistently) that a Grothendieck space $C(K)$ does
not contain a copy of $\ell_1(\omega_1)$. Such an example was constructed by Talagrand in \cite{talagrand}
under {\sf CH}.

\section{Negative results for densities of cofinality bigger than $\omega_1$}\label{section-bigger}

\begin{theorem}\label{bigger-l1} Let $\kappa$ be a cardinal
satisfying $cf(\kappa)>\omega_1$. If $X$ is a Banach space of density $\kappa$
containing an isomorphic copy of  $\ell_1(\omega_1)$, then
$X$ does not admit an overcomplete set.
\end{theorem}
\begin{proof} Let $T:\ell_1(\omega_1)\rightarrow L_\infty(\{0,1\}^{\omega_1})$
be a linear bounded operator such that $T(1_{\{\alpha\}})=x_\alpha$
for all $\alpha<\omega_1$, where $x_\alpha$ is the $\alpha$-th coordinate function. 
It exists by the universal property of $\ell_1(\omega_1)$.
Let $S: X\rightarrow  L_\infty(\{0,1\}^{\omega_1})$ be an extension of $T$ obtained using the injectivity of 
the space $L_\infty(\{0,1\}^{\omega_1})$. For each $\alpha<\omega_1$ 
consider the subspace $Y_\alpha$ of $L_\infty(\{0,1\}^{\omega_1})$ consisting
of all elements which depend on coordinates in below $\alpha$. 
The union $\bigcup_{\alpha<\omega_1}Y_\alpha$ is the entire space and
$x_\alpha\not \in Y_\alpha$ for any $\alpha<\omega_1$. It follows that
$S^{-1}[Y_\alpha]$s  for $\alpha<\omega_1$ form a strictly increasing sequence of proper subspaces of
$X$. So Lemma \ref{unions} implies that $X$ does not admit an overcomplete set.

\end{proof}

\begin{theorem}\label{wld-dual} Let $\kappa$ be a cardinal
satisfying $cf(\kappa)>\omega_1$. Suppose that $X$
is a Banach space of density $\kappa$ such that $X^*$
contains a nonseparable WLD subspace. Then $X$ does not admit
an overcomplete set.
\end{theorem}
\begin{proof}
Let $Y$ be a nonseparable WLD space and let $T: Y\rightarrow X^*$ be an  isomorphism onto its
image.
We may assume that the density of $Y$ is $\omega_1$ as subspaces of WLD spaces are WLD (Corollary 9 of \cite{gonzalez}).
Let $\{y_\alpha: \alpha<\omega_1\}$ be a linearly dense subset of $Y$ such that
each element $y^*\in Y^*$ is countably supported
 by $\{y_\alpha: \alpha<\omega_1\}$ i.e., $s(y^*)=\{\alpha<\omega_1: y^*(y_\alpha)\not=0\}$
is at most countable. The existence of such a linearly dense set is equivalent to being WLD
by Theorem 7 of \cite{gonzalez}.

Let $J: X\rightarrow X^{**}$ be the canonical isometric embedding and 
$$S= T^*\circ J: X\rightarrow Y^*.$$
Note that $Y^*=\bigcup_{\alpha<\omega_1}Z_\alpha$ where 
$$Z_\alpha=\{y^*\in Y^*: s(y^*)\subseteq\alpha\}=\bigcap_{\beta\geq\alpha}ker(J(y_\beta))$$
is a norm-closed subspace of $Y^*$ for each $\alpha<\omega_1$.
 So to use Lemma \ref{unions} and conclude that $X$ does not admit an overcomplete set it is enough
to note that  the subspaces $S^{-1}[Z_\alpha]$ are proper for $\alpha<\omega_1$.
 To do so choose
$x_\alpha\in X$ such that $T(y_\alpha)(x_\alpha)\not=0$.
We get $S(x_\alpha)(y_\alpha)=T(y_\alpha)(x_\alpha)\not=0$, so $s(S(x_\alpha))\not\subseteq\alpha$
and so $x_\alpha\not\in S^{-1}[Z_\alpha]$
as required.
\end{proof}

\begin{theorem}\label{bigger-groth} Let $\kappa$ be a cardinal
satisfying $cf(\kappa)>\omega_1$. If $X$
is a nonreflexive Grothendieck space of density $\kappa$, then it does
not admit an overcomplete set.
\end{theorem}
\begin{proof}
 In \cite{haydon-groth} R. Haydon proved  that if $X$ is a
  nonreflexive Grothendieck space, then $X^*$ 
contains an isomorphic copy of  $L_1(\{0,1\}^{\mathfrak p})$.
As $\mathfrak p\geq \omega_1$, it is a nonseparable WLD subspace. So
Theorem \ref{wld-dual} can be applied.

\end{proof}

\begin{theorem}\label{bigger-scat}  Let $\kappa$ be a cardinal
satisfying $cf(\kappa)>\omega_1$. Let $K$ be  a scattered
compact space of cardinality $\kappa$ (equivalently $C(K)$ has density $\kappa$).
Then the Banach space $C(K)$ does not admit an
overcomplete set.
\end{theorem}
\begin{proof}  Let $\{f_\xi: \xi<\kappa\}\subseteq C(K)$.
We will show that there is $A\subseteq \kappa$ such that $|A|=\kappa$
and $\{f_\xi: \xi\in A\}$ does not separate points of $K$. 
Let $X\subseteq K$ be of cardinality $\omega_1$.
A continuous image of a compact scattered space is scattered.
So $f_\xi[X]\subseteq \R$ is countable. It follows that for
every $\xi<\kappa$ there is a pair $\{x, y\}\in [X]^2$ such that
$f_\xi(x)=f_\xi(y)$. As there are $\omega_1$ pairs in $[X]^2$ and
$\kappa$ has cofinality bigger than $\omega_1$, we conclude that
there are $x, y\in X$ such that $A=\{\xi<\kappa: f_\xi(x)=f_\xi(y)\}$
has cardinality $\kappa$. Then $\{f_\xi: \xi\in A\}$ does not separate points of $K$, as required.
\end{proof}

\noindent{\bf Remark.} Note that in the result above
we show that for every $D\subseteq C(K)$ which is linearly dense there is
$D'\subseteq D$ of the same cardinality which does not generate $C(K)$ even as an algebra.
This is a stronger property than not being overcomplete. One notes that this property
behaves differently than the property of not being overcomplete.
For example, under {\sf CH} the algebra $\ell_\infty$ contains $D$ such that $D'\subseteq D$
generates $\ell_\infty$ as an algebra for every uncountable $D'\subseteq D$. For this
represent $\ell_\infty$ as an increasing sequence of algebras $C(K_\alpha)$
for $\alpha<\omega_1$, where $K_\alpha$s are totally disconnected and metrizable. Choose
$f_\alpha\in \ell_\infty$ which separates all points of $K_\alpha$, then
$D=\{f_\alpha: \alpha<\omega_1\}$ works. On the other hand
it is consistent that for any set $\{T_\alpha: \alpha<\mathfrak c\}\subseteq \B(\ell_2)$
which generates  a subalgebra of $\B(\ell_2)$ of density $\mathfrak c$ there is
a subset $A\subseteq \mathfrak c$ of cardinality $\mathfrak c$ such that
$T_\alpha$ is not in the algebra generated by $\{T_\beta: \beta\in A\setminus\{\alpha\}\}$
for any $\alpha\in A$ (\cite{irr}). This applies to $\ell_\infty\subseteq \B(\ell_2)$.

After this paper
has been completed and submitted it was proved in  \cite{hyperplanes}  that
it is consistent (with any possible size of the continuum) that no  Banach space
 of density $\kappa$ with $cf(\kappa)>\omega_1$ admits an overcomplete set (Theorem 8
of \cite{hyperplanes}).

\section{Final remarks and questions}

The main topic of this paper is to determine which nonseparable Banach spaces admit overcomplete sets.
The first natural question would be to determine
the densities of Banach spaces which admit overcomplete sets as we now know that there
are such nonseparable spaces. Here the most restrictive possibility would be 
to prove the positive answer to the following:

\begin{question} Can one prove in ZFC that if a Banach space admits an overcomplete set,
then $dens(X)\leq\omega_1$?
\end{question}

Under {\sf CH} the results of \cite{russo} (in this paper Theorem \ref{russo} (4)) imply that no
Banach space of density $\omega_n$ for $n>1$ can admit an overcomplete set. As mentioned at the end of section 6 after this paper
has been completed and submitted it was proved in Theorem 3 of \cite{hyperplanes}  that
it is consistent (with any possible size of the continuum) that every nonseparable Banach space
is a union of $\omega_1$ of its hyperplanes, consequently it is consistent that
no Banach space of density $\kappa$ with $cf(\kappa)>\omega_1$ admits an overcomplete set (Theorem 8
of \cite{hyperplanes}).
No
consistent, example of an overcomplete set in a Banach space of density $\omega_2$
is known. On the other hand it is also left open if in ZFC there can be overcomplete sets
of singular cardinalities (cf. Question 1.2 (ii) of \cite{russo}). In
the light of the results of Section \ref{section-bigger} this is especially
interesting for cardinals of cofinalities $\omega$ or $\omega_1$:
\begin{question} Is there (in ZFC or consistently) a Banach space of singular density which
admits an overcomplete set?
\end{question}
Note that it is not uncommon that completely distinct phenomena take place in Banach
spaces of some singular densities (see \cite{argyros-uni}, \cite{argyros-lms}).
Recall that the matter  of
densities of Banach spaces which admit overcomplete sets was completely settled for WLD spaces by
Corollary \ref{wld-iff}.

For the lack of examples of Banach spaces  with overcomplete sets of densities above $\omega_1$
the next natural question is to characterize Banach spaces of density $\omega_1$ 
which admit an overcomplete set.
Such characterization as being WLD was obtained in Corollary \ref{unconditional} for Banach spaces of density $\omega_1$ with an
unconditional basis.
 By Corollary
\ref{cor-psi} such a characterization for the general class of Banach spaces of density
$\omega_1$  cannot be  in terms of properties which do not change
when we pass from one model of set theory to another. However we can also aim at characterizations under
additional set-theoretic hypotheses. For example under {\sf CH} having
the dual of cardinality not bigger than $\omega_1$ can serve as such a characterization by 
the results of \cite{russo} if we have the positive answer to the following:

\begin{question} Does {\sf CH} imply that no Banach space  whose dual  has
density bigger than $\omega_1$ admits an overcomplete set?
\end{question}

As noted in the Introduction under the negation of {\sf CH} there are
nonseparable Banach spaces whose duals have cardinality $\mathfrak c>\omega_1$ 
which admit overcomplete sets. But we do not know what is ZFC answer to the following

\begin{question} Is it true that  no  Banach space whose  dual has density
bigger than $\mathfrak c$ admits an overcomplete set?
\end{question}

On the other hand taking into account Theorem \ref{nonmonolithic} one could 
hope for another consistent characterization of Banach spaces of density $\omega_1$ which
admit overcomplete sets as spaces with monolithic dual balls in the weak$^*$ topology:

\begin{question} Does {\sf MA} and the negation of {\sf CH} imply
that every Banach space whose dual ball is monolithic in the weak$^*$ topology admit
an overcomplete set?
\end{question} 

Actually the above question seems open even in ZFC. Here the key case could be the following

\begin{question} Does $C(K)$ admit an overcomplete set
if $K$ is the ladder system space of \cite{ladder}? 
\end{question}

The results of section 4 shed some light on the relation between the existence of
overcomplete sets and the cardinal characteristics of the continuum in the sense of \cite{blass}.
Recall that {\sf MA} for partial orders having precaliber $\omega_1$ implies
that $\mathfrak p=\mathfrak  c$ and ${\mathfrak{add}}(\mathcal M)=\mathfrak c$ (2.15 and 2.20 of
\cite{kunen}). So by Theorem \ref{ma-precaliber} the statement  that every Banach space whose dual has density $\omega_1$
admits an overcomplete set is consistent with all cardinal invariants in van Douwen's
diagram being $\mathfrak c$ and all cardinal invariants in the Cicho\'n's diagram above or
equal to ${\mathfrak{add}}(\mathcal M)$
being $\mathfrak c$. This is somewhat surprising because, clearly overcomplete sets cannot exist
in Banach spaces of density $\omega_1$ if we can cover their sets of cardinality $\omega_1$ by countably many hyperplanes.  Moreover, hyperplanes are examples of nowhere dense sets but 
${\mathfrak{cov}}(\mathcal M)$ is the minimal cardinality of a subset of $\R$ which cannot be covered by the union of countably many nowhere dense sets and has value $\mathfrak c>\omega_1$ under the above version of
Martin's axiom.

On the other hand Pawlikowski proved that {\sf MA} for 
partial orders having precaliber $\omega_1$ is 
consistent with ${\mathfrak{cov}}(\mathcal N)=\omega_1<\mathfrak c=\omega_2$ (\cite{pawlikowski}).
Moreover, the main technical result of Section 4 that is Lemma \ref{hard-density} consists of
a construction of an element of $\ell_1$ and it is known that the structure of $\ell_1$ is
related to cardinal characteristics of the measure rather than category (e.g. Theorem 2.3.9 of \cite{tomek}, \cite{tomek-tams}). So it is natural to ask the following:

\begin{question} Is it true that Banach spaces $X$ with separable dual balls in the weak$^*$ topology
satisfying  $dens(X)<{\mathfrak{cov}}(\mathcal N)$  and $cf(dens(X))>\omega$ do not admit overcomplete sets?
\end{question}

There are also some natural open questions left which concern more particular classes
for Banach spaces like the following three questions:

\begin{question}\label{q-ordinals} Can one prove  in ZFC that the Banach spaces
$C([0, \xi])$ for all ordinals $\xi<\omega_2$ admit overcomplete sets?
\end{question}

Partial results related to Question \ref{q-ordinals} are Theorem \ref{main-positive} (2) and 
result of \cite{russo} (Theorem \ref{russo} (3) in this paper). They cover all
remaining ordinals as $C([0, \xi])$ admits a fundamental biorthogonal
system $(1_{[0,\eta]}, \delta_{\eta}-\delta_{\eta+1})_{\eta<\xi}$ and $1_{[0,\xi]}, \delta_\xi$.

\begin{question}\label{q-groth} Can one prove in ZFC that no nonreflexive Grothendieck space admits an overcomplete set?
\end{question}

 Partial results related to Question \ref{q-groth} are  Corollary \ref{argyros-ma},
Theorem \ref{ck-groth} and 
Theorem \ref{bigger-groth}.
Although the above negative ZFC results do not imply the positive answer to Question \ref{q-groth}  the exotic
$C(K)$s with the Grothendieck property which we know from the literature 
are covered by our results. For example
examples of Brech (\cite{brech}), Fajardo (\cite{rogerio}) and  Sobota and Zdomskyy 
(\cite{sobota-zdomski}) contain
$\ell_1(dens(C(K))$ so do not admit an overcomplete set by Theorem \ref{main-negative}. 
Talagrand's example from \cite{talagrand}  does not contain $\ell_1(\omega_1)$ but
is covered by  Theorem \ref{ck-groth}.
Haydon's example of \cite{haydon} is induced by
a Boolean algebra which  satisfies the subsequential completeness property and
so has the weak subsequential separation property of \cite{com-shelah}. Consequently 
by the results of \cite{com-shelah} it contains  an independent family of size $\mathfrak c$, which
yields $\ell_1(\mathfrak c)$ and implies that there is no overcomplete set by Theorem \ref{main-negative}.

\begin{question}\label{q-scattered} Can one prove in ZFC that no Banach space of
the form $C(K)$, where $K$ is scattered compact Hausdorff space of cardinality
bigger than $\omega_1$ admits an overcomplete set?
\end{question}

This would improve Theorem \ref{bigger-scat}.

Looking at our results, it seems also intriguing if the property of admitting
an overcomplete set behaves well with respect to canonical operations on Banach spaces.
For example: 

\begin{question} Is the admitting overcomplete sets a hereditary property with respect 
to closed subspaces of the same
density?
\end{question}

This would be a generalization of the main part of Theorem \ref{main-negative}. We also do not know
the answer to the following: 

\begin{question} Does the direct sum of two Banach spaces that admit overcomplete sets admit an overcomplete set?
In particular does $X\oplus \R$ admit an overcomplete sets if $X$ does so?
\end{question}

A positive answer to the the above question  would simplify the conclusion of 
Theorem \ref{yes-zfc}. At the end we note that admitting an overcomplete set is  not (at least consistently)
a three space property: the space $C(K)$ of Corollary \ref{cor-psi} 
satisfies $C(K)/c_0\equiv c_0(\omega_1)$
but consistently does not admit an overcomplete set, while $c_0$ and $c_0(\omega_1)$
admit overcomplete sets by Theorem \ref{klee} and Theorem \ref{main-positive} (1) (c).

\bibliographystyle{amsplain}

\end{document}